\documentclass[12pt,a4paper]{article}
\usepackage{mathptmx}
\usepackage[T1]{fontenc}
\usepackage[a4paper]{geometry}
\usepackage{color}
\usepackage[pdftex]{graphicx}
\usepackage{xypic}
\usepackage{latexsym}
\usepackage{amsthm}
\usepackage{amsmath}
\usepackage{amssymb}

\begin{document}

\newcommand{\To}{\longrightarrow}
\newcommand{\h}{\mathcal{H}}
\newcommand{\s}{\mathcal{S}}
\newcommand{\A}{\mathcal{A}}
\newcommand{\K}{\mathcal{K}}
\newcommand{\B}{\mathcal{B}}
\newcommand{\W}{\mathcal{W}}
\newcommand{\M}{\mathcal{M}}
\newcommand{\Lom}{\mathcal{L}}
\newcommand{\T}{\mathcal{T}}
\newcommand{\F}{\mathcal{F}}

\newtheorem{definition}{Definition}[section]
\newtheorem{defn}[definition]{Definition}
\newtheorem{lem}[definition]{Lemma}
\newtheorem{prop}[definition]{Proposition}
\newtheorem{thm}[definition]{Theorem}
\newtheorem{cor}[definition]{Corollary}
\newtheorem{cors}[definition]{Corollaries}
\newtheorem{example}[definition]{Example}
\newtheorem{examples}[definition]{Examples}
\newtheorem{rems}[definition]{Remarks}
\newtheorem{rem}[definition]{Remark}
\newtheorem{notations}[definition]{Notations}
\theoremstyle{remark}
\theoremstyle{remark}
\theoremstyle{remark}

\theoremstyle{notations}
\theoremstyle{remark}
\theoremstyle{remark}
\theoremstyle{remark}
\newtheorem{dgram}[definition]{Diagram}
\theoremstyle{remark}
\newtheorem{fact}[definition]{Fact}
\theoremstyle{remark}
\newtheorem{illust}[definition]{Illustration}
\theoremstyle{remark}
\theoremstyle{definition}
\newtheorem{question}[definition]{Question}
\theoremstyle{definition}
\newtheorem{conj}[definition]{Conjecture}

\title{\textbf{SS-Injective Modules and Rings}}

\author{\textbf{Adel Salim Tayyah} \\ Department of Mathematics, College of Computer Science and \\Information Technology, Al-Qadisiyah University,  Al-Qadisiyah, Iraq \\Email: adils9888@gmail.com  \\ \\\textbf{Akeel Ramadan Mehdi}\\ Department of Mathematics, College of Education, \\Al-Qadisiyah University, P. O. Box 88, Al-Qadisiyah, Iraq \\Email: akeel\_math@yahoo.com}

\date{\today}

\maketitle

\begin{abstract} \,We introduce and investigate ss-injectivity as a  generalization of both soc-injectivity and small injectivity. A module $M$ is said to be  ss-$N$-injective \,(where $N$ is a module) \,if \, every \, $R$-homomorphism from a semisimple small submodule of $N$ into $M$ extends to $N$. A module $M$  is said to be ss-injective (resp. strongly ss-injective),  if $M$ is ss-$R$-injective (resp. ss-$N$-injective for every right $R$-module $N$).
 Some characterizations and properties of (strongly) ss-injective  modules and rings are given. Some results of Amin, Yuosif and Zeyada on soc-injectivity are extended to ss-injectivity.  Also, we provide some new characterizations of universally mininjective rings, quasi-Frobenius rings, Artinian rings and semisimple rings.
\end{abstract}

$\vphantom{}$


$\vphantom{}$


$\vphantom{}$

\noindent \textbf{$\ast$} The results of this paper will be part of a MSc thesis of the first  author, under the supervision of the second author at the University of Al-Qadisiyah.

$\vphantom{}$

\section{Introduction}
Throughout this paper, $R$ is an associative ring with identity, and all modules are unitary $R$-modules. For a right $R$-module $M$, we write soc$(M)$, $J(M)$,  $Z(M)$,  $Z_{2}(M)$, $E(M)$ and End$(M)$ for  the socle, the Jacobson radical, the singular submodule,   the second singular submodule, the injective hull and the endomorphism ring of $M$, respectively. Also, we use $S_{r}$, $S_{\ell}$, $Z_{r}$, $Z_{\ell}$, $Z_{2}^{r}$ and $J$ to indicate the right socle, the left socle, the right singular ideal, the left singular ideal, the right second singular ideal, and the Jacobson radical of $R$, respectively. For a submodule $N$ of $M$, we write $N\subseteq^{ess}M$, $N\ll M$, $N\subseteq^{\oplus}M$, and $N\subseteq^{max}M$ to indicate that $N$ is an essential submodule, a small submodule, a direct summand, and a maximal submodule of $M$, respectively. If $X$ is a subset of
a right $R$-module $M$, the right (resp. left) annihilator of $X$ in $R$  is denoted by $r_{R}(X)$ (resp. $l_{R}(X)$). If $M=R$, we write $r_{R}(X)=r(X)$ and
$l_{R}(X)=l(X)$.

Let $M$ and $N$ be right $R$-modules, $M$ is called soc-$N$-injective if every $R$-homomorphism from the soc$(N)$ into $M$ extends to $N$.
A right $R$-module $M$ is called soc-injective, if $M$ is soc-$R$-injective. A right $R$-module $M$ is called strongly soc-injective, if $M$ is soc-$N$-injective for all right $R$-module $N$ \cite{2AmYoZe05}

Recall that a right $R$-module $M$ is called mininjective \cite{14NiYo97} (resp. small injective \cite{19ThQu09}, principally small injective \cite{20Xia11}) if every $R$-homomorphism from any simple (resp. small, principally small) right ideal to $M$ extend to $R$. A ring is called right mininjective (resp. small injective, principally small injective) ring, if it is right mininjective (resp. small injective, principally small injective) as right $R$-module. A ring $R$ is called right Kasch if every simple right $R$-module embeds in $R$ (see for example \cite{15NiYu03}. Recall that a ring $R$ is called semilocal if $R/J$ is a  semisimple \cite{11Lom99}. Also, a ring  $R$ is said to be right perfect if every right $R$-module has a projective cover. Recall that a ring $R$ is said to be quasi-Frobenius (or $QF$) ring if it is right (or left) artinian and right (or left) self-injective; or equivalently, every injective right $R$-module is projective.

In this paper, we introduce and investigate the notions of ss-injective and strongly ss-injective modules and rings. Examples are given to show that the (strong) ss-injectivity is distinct from that of mininjectivity, principally small injectivity, small injectivity, simple J-injectivity, and (strong) soc-injectivity. Some characterizations and properties of (strongly) ss-injective  modules and rings are given.

W. K. Nicholson and M. F. Yousif in \cite{14NiYo97} introduced the notion of universally mininjective ring, a ring $R$  is called right universally mininjective if $S_{r}\cap J=0$. In Section 2, we show that $R$  is a right universally mininjective ring if and only if  every simple right $R$-module is ss-injective. We also prove that if $M$  is a projective right $R$-module, then every quotient of an ss-$M$-injective right $R$-module is ss-$M$-injective
if and only if  every sum of two ss-$M$-injective submodules of a right $R$-module is ss-$M$-injective if and only if Soc$(M)\cap J(M)$ is projective. Also, some results are given in terms of ss-injectivity modules. For example, every simple singular right $R$-module is ss-injective implies that $S_{r}$ projective and $r(a)\subseteq^{\oplus}R_{R}$ for all $a\in S_{r}\cap J$, and if $M$  is a finitely generated right $R$-module, then Soc$(M)\cap J(M)$ is finitely generated if and only if every direct sum of ss-$M$-injective right $R$-modules is ss-$M$-injective if and only if every direct sum of \emph{$\mathbb{N}$} copies of ss-$M$-injective right $R$-module is ss-$M$-injective.

In Section 3, we show that a right $R$-module $M$  is strongly ss-injective if and only if every small submodule $A$  of a right $R$-module $N$, every $R$-homomorphism $\alpha:A\longrightarrow M$ with $\alpha(A)$ semisimple extends to $N$. In particular, $R$  is semiprimitive if every simple right $R$-module is strongly ss-injective, but not conversely. We also prove that if $R$  is a right perfect ring, then a right $R$-module $M$  is strongly
soc-injective  if and only if $M$  is strongly ss-injective. A results (\cite[Theorem 3.6 and Proposition 3.7]{2AmYoZe05})   are extended. We prove that a ring  $R$  is right artinian if and only if every direct sum of strongly ss-injective right $R$-modules is injective, and $R$
is $QF$  ring if and only if every strongly ss-injective right $R$-module is projective.

In Section 4, we extend the results (\cite[Proposition 4.6 and Theorem 4.12]{2AmYoZe05})  from a soc-injective ring to an ss-injective
ring (see Proposition~\ref{Proposition:(4.14)}  and Corollary~\ref{Corollary:(4.18)}).

In Section 5, we show that a ring $R$  is $QF$  if and only if $R$  is strongly ss-injective and right noetherian with essential right socle if and only if $R$  is strongly ss-injective, $l(J^{2})$ is countable generated left ideal, $S_{r}\subseteq^{ess}R_{R}$, and the chain $r(x_{1})\subseteq r(x_{2}x_{1})\subseteq...\subseteq r(x_{n}x_{n-1}...x_{1})\subseteq...$ terminates for every infinite sequence $x_{1},x_{2},...$ in $R$  (see Theorem~\ref{Theorem:(5.10)}   and Theorem~\ref{Theorem:(5.12)}). Finally, we prove that a ring $R$  is $QF$  if and only if $R$  is strongly left and right ss-injective, left Kasch, and $J$  is left $t$-nilpotent (see Theorem~\ref{Theorem:(5.16)}), extending a result of  I. Amin, M. Yousif and N. Zeyada \cite[Proposition 5.8]{2AmYoZe05}  on strongly soc-injective rings.

General background materials can be found in \cite{3AnFu74}, \cite{9Kas82}  and \cite{10Lam99}.

\section{SS-Injective Modules}

\begin{defn}\label{Definition:(2.1)(a)}
Let \textit{N} be a right \textit{R}-module. A right \textit{R}-module
\textit{M} is said to be ss-\textit{N}-injective, if for any semisimple
small submodule \textit{K} of \textit{N}, any right \textit{R}-homomorphism
\textit{\textcolor{black}{$f:K{\color{black}\longrightarrow}M$}}\textit{
}extends to \textit{N}. A module \textit{M} is said to be ss-\textit{quasi}-injective
if \textit{M} is ss-\textit{M}-injective. \textit{M} is said to be
ss-injective if \textit{M} is ss-\textit{R}-injective. A ring \textit{R}
is said to be right ss-injective if the right \textit{R}-module $\mathit{{\color{black}R}_{{\color{black}R}}}$
is ss-injective.
\end{defn}

\begin{defn}\label{Definition:(2.1(b))}
A right $R$-module $M$  is said to be strongly ss-injective if $M$ is ss-$N$-injective, for all right $R$-module $N$. A ring $R$ is said to be strongly right ss-injective if the right $R$-module $R_{R}$ is strongly ss-injective.
\end{defn}

\begin{example}\label{example:(2.2)}
\noindent \emph{(1) Every soc-injective  module is ss-injective, but
not conversely (see Example~\ref{Example:(5.8)}).}

\noindent\emph{(2) Every  small injective module is ss-injective, but
not conversely (see Example~\ref{Example:(5.6)}).}

\noindent\emph{(3)  Every $\mathbb{Z}$-module is ss-injective. In fact, if $M$ is a $\mathbb{Z}$-module, then $M$ is small injective (by \cite[Theorem 2.8]{19ThQu09} and hence it is ss-injective.}

\noindent\emph{(4)  The two classes of principally small injective rings and ss-injective
rings are different (see \cite[Example 5.2]{15NiYu03}, Example~\ref{Example:(4.4)} and Example~\ref{Example:(5.6)}).}

\noindent\emph{(5)  Every strongly soc-injective module is strongly ss-injective,
but not conversely (see Example~\ref{Example:(5.8)}).}

\noindent\emph{(6)  Every strongly ss-injective module is ss-injective, but not conversely
(see Example~\ref{Example:(5.7)}).}
\end{example}

\begin{thm}\label{Theorem:(2.3)} The following statements hold:

\noindent (1) Let  $N$  be a right $R$-module and let ${\color{black}\left\{ M_{i}:i\in I\right\} }$ be a family of right $R$-modules. Then the direct product $\prod_{{\scriptstyle {\scriptscriptstyle i\in I}}}M_{i}$ is ss-$N$-injective if and only if each $\mathit{M_{i}}$ is ss-$N$-injective, for all $\mathit{i\in I}$.

\noindent (2) Let $M$, $N$  and  $K$ be right $R$-modules with $K{\color{black}\subseteq}N$. If $M$  is ss-$N$-injective, then $M$  is ss-$K$-injective.

\noindent (3) Let $M$, $N$  and $K$  be right $R$-modules with ${\color{black}M{\color{black}\cong}N}$. If $M$   is ss-$K$-injective, then $N$  is ss-$K$-injective.

\noindent (4) Let $M$, $N$  and $K$ be right $R$-modules with ${\color{black}K{\color{black}\cong}N}$. If $M$ is ss-$K$-injective, then $M$  is ss-$N$-injective.

\noindent (5) Let $M$, $N$  and $K$  be right $R$-modules with $N$  is a direct summand of $M$. If $M$   is ss-$K$-injective, then  $N$   is ss-$K$-injective.
\end{thm}

\begin{proof} Clear.
\end{proof}

\begin{cor}\label{Corollary:(2.4)}
\noindent(1) If $N$  is a right $R$-module, then a finite direct sum of ss-$N$-injective modules is again ss-$N$-injective. Moreover, a finite direct sum of ss-injective (resp. strongly ss-injective) modules is again ss-injective (resp. strongly ss-injective).

\noindent(2) A direct summand of an ss-quasi-injective (resp. ss-injective,
strongly ss-injective) module is again ss-quasi-injective
(resp. ss-injective, strongly ss-injective).
\end{cor}
\begin{proof}
(1) By taking the index $I$ to be a finite set and applying Theorem~\ref{Theorem:(2.3)}(1).

(2) This follows from Theorem~\ref{Theorem:(2.3)}(5).
\end{proof}

\begin{lem}\label{Lemma:(2.5)}
Every ss-injective right $R$-module is right mininjective.
\end{lem}
\begin{proof}
Let $I$ be a simple right ideal of $R$. By \cite[Lemma 3.8]{16Pas04} we have that either $I$ is nilpotent or a direct summand of $R$. If $I$ is a nilpotent, then $I\subseteq J$ by \cite[Corollary 6.2.8]{6Bla11}  and hence $I$ is a semisimple small right ideal of $R$. Thus every ss-injective right $R$-module is right mininjective.
\end{proof}

It easy to prove the following proposition.

\begin{prop}\label{Proposition:(2.6)} Let $N$  be a right $R$-module. If $J(N)$ is a small submodule of $N$, then  a right $R$-module   $M$  is ss-$N$-injective if and only if any $R$-homomorphism $f:soc(N)\cap J(N){\color{black}\longrightarrow M}$  extends to $N$.
\end{prop}

\begin{prop}\label{Proposition:(2.8)}
Let $N$  be a right $R$-module and  $\left\{ {\color{black}A_{i}:i=1,2,...,n}\right\} $ be a family of finitely generated right $R$-modules. Then $N$ is ss-$\overset{{\scriptscriptstyle n}}{\underset{{\scriptscriptstyle i=1}}{\bigoplus}}A_{i}$-injective if and only if $N$  is ss-$\mathit{A_{i}}$-injective, for all $\mathit{i}=1,2,...,n$.
\end{prop}
\begin{proof}
($\Rightarrow$) This follows from Theorem~\ref{Theorem:(2.3)}((2),(4)).

$\!\!\!\!\!\!\!\!$($\Leftarrow$) By \cite[Proposition (I.4.1) and Proposition (I.1.2)]{5BiKeNe82} we have soc($\overset{{\scriptscriptstyle n}}{\underset{{\scriptscriptstyle i=1}}{\bigoplus}}A_{i})\cap J(\overset{{\scriptscriptstyle n}}{\underset{{\scriptscriptstyle i=1}}{\bigoplus}}A_{i})=($soc$\,\cap \,J)(\overset{{\scriptscriptstyle n}}{\underset{{\scriptscriptstyle i=1}}{\bigoplus}}A_{i})$ $=\mathit{\overset{{\scriptscriptstyle n}}{\underset{{\scriptscriptstyle i=1}}{\bigoplus}}}($soc$\,\cap\, J)(\mathit{A_{i}})=\overset{{\scriptscriptstyle n}}{\underset{{\scriptscriptstyle i=1}}{\bigoplus}}($soc$(\mathit{A_{i}})\,\cap\,J(\mathit{A_{i}}))$.
For $j=1,2,...,n$,  consider the following diagram:

\[
\xymatrix{
K_{j}=soc(\mathit{A_{j}})\,\cap\,J(\mathit{A_{j}}) \,\ar[d]_{i_{K_{j}}}\ar@{^{(}->}[r]^{\qquad\qquad i_{2}} \ar[r] & A_{j} \ar[d]^{i_{A_{j}}} \\
\overset{{\scriptscriptstyle n}}{\underset{{\scriptscriptstyle i=1}}{\bigoplus}}(soc(\mathit{A_{i}})\,\cap\,J(\mathit{A_{i}})) \,\ar[d]_{f}\ar@{^{(}->}[r]^{\qquad\qquad i_{1}} \ar[r] &{\overset{{\scriptscriptstyle n}}{\underset{{\scriptscriptstyle i=1}}{\bigoplus}}\mathit{A_{i}}}\\N
}
\]

\noindent where  \,$i_{1}$, \,$i_{2}$ \, are inclusion maps and \,  $i_{K_{j}}$, \, $i_{A_{j}}$ \, are injection maps. \, By hypothesis, \, there exists an $R$-homomorphism $h_{j}:\mathit{A_{j}\longrightarrow N}$ such that $\mathit{h_{j}\circ i_{2}=f\circ i_{K_{j}}}$,
also there exists exactly one homomorphism $h:\overset{{\scriptscriptstyle n}}{\underset{{\scriptscriptstyle i=1}}{\bigoplus}}A_{i}\longrightarrow N$
satisfying $\mathit{h_{j}}=h\circ i_{A_{j}}$ by \cite[Theorem 4.1.6(2)]{9Kas82}. Thus $\mathit{f\circ i_{K_{j}}}=h_{j}\circ i_{2}=h\circ i_{A_{j}}\circ i_{2}=h\circ i_{1}\circ i_{K_{j}}$ for all $\mathit{j}=1,2,...,n$. Let $(\mathit{a_{1},a_{2},}...,a_{n})$$\in\overset{{\scriptscriptstyle n}}{\underset{{\scriptscriptstyle i=1}}{\bigoplus}}($soc$(A_{i})\cap J(A_{i}))$,
thus $\mathit{a_{j}\in}$soc$(A_{j})\cap J(A_{j})$, for all $\mathit{i}=1,2,...,n$
and, $f(\mathit{a_{1},a_{2},}...,a_{n})=f(\mathit{i_{K_{1}}}(a_{1}))+$$f(\mathit{i_{K_{2}}}(a_{2}))+...+$$f(\mathit{i_{K_{n}}}(a_{n}))=(\mathit{h\circ i_{1}})(\mathit{a_{1},a_{2},}...,a_{n})$. Thus $\mathit{f=h\circ i_{1}}$ and the proof is complete.
\end{proof}

\begin{cor}\label{Corollary:(2.10)} Let $M$
be a right $R$-module and $1=\mathit{e_{1}}+e_{2}+...+e_{n}$ in $R$ such that $\mathit{e_{i}}$ are orthogonal idempotent. Then $M$ is ss-injective if and only if $M$ is ss-$\mathit{e_{i}R}$-injective for every $\mathit{i=1,2,}...,n$.

\noindent (2) For idempotents $e$ and $f$ of $R$. If $\mathit{eR\cong f}R$ and $M$   is ss-$\mathit{eR}$-injective, then $M$  is ss-$\mathit{fR}$-injective.
\end{cor}
\begin{proof}
\noindent (1) From \cite[Corollary 7.3]{3AnFu74}, we have $\mathit{R=\overset{{\scriptscriptstyle n}}{\underset{{\scriptscriptstyle i=1}}{\bigoplus}}}e_{i}R$, thus it follows from Proposition~\ref{Proposition:(2.8)}  that $M$ is ss-injective if and only if $M$  is ss-$\mathit{e_{i}}R$-injective for all 1$\leq i\leq n$.

(2) This follows from Theorem~\ref{Theorem:(2.3)}(4).
\end{proof}

\begin{prop}\label{Proposition:(2.9)}  A right $R$-module $M$ is ss-injective if and only if $M$ is ss-$P$-injective, for every finitely generated projective right $R$-module $P$.
\end{prop}
\begin{proof} ($\Rightarrow$) Let $M$  be an ss-injective $R$-module, thus it follows from Proposition~\ref{Proposition:(2.8)} that $M$ is ss-$\mathit{R^{n}}$-injective for any $\mathit{n\in\mathbb{\mathbb{Z^{\dotplus}}}}$. Let $P$  be a finitely generated projective $R$-module, thus by \cite[Corollary 5.5]{1AdWe92}, we have that $P$ is a direct summand of a module isomorphic to $\mathit{R^{m}}$ for some $\mathit{m\in\mathbb{Z^{\dotplus}}}$. Since $M$  is ss-$\mathit{R^{m}}$-injective, thus $M$
is ss-$P$-injective by Theorem~\ref{Theorem:(2.3)}((2),(4)).

($\Leftarrow$) By the fact that $R$  is projective.
\end{proof}

\begin{prop}\label{Proposition:(2.11)} The following statements are equivalent for a right $R$-module $M$.

\noindent(1) Every right $R$-module is ss-$M$-injective.

\noindent(2) Every simple submodule of $M$ is ss-$M$-injective.

\noindent(3) \emph{soc}$(M)\cap J(M)=0$.
\end{prop}
\begin{proof}
(1) $\Rightarrow$ (2) and (3) $\Rightarrow$ (1) are obvious.

(2) $\Rightarrow$ (3) Assume that soc$(M)\cap J(M)\neq 0$, thus soc$(M)\cap J(M)=\underset{i\in I}{\bigoplus}x_{i}R$ where $x_{i}R$ is a simple small submodule of $M$, for each $i\in I$. Therefore,  $x_{i}R$ is ss-$M$-injective for each $i\in I$ by hypothesis. For any $i\in I$, the inclusion map from $x_{i}R$ to $M$ is split, so we have that $x_{i}R$ is a direct summand of $M$.   Since $x_{i}R$ is small submodule of $M$, thus $x_{i}R=0$ and hence $x_{i}=0$ for all $i\in I$ and this a contradiction.

\end{proof}

\begin{lem}\label{Lemma:(2.13)} Let $M$  be an ss-quasi-injective right $R$-module and $S=\emph{End}(M_{R})$, then the following statements hold:

\noindent(1) $l_{M}r_{R}(m)=Sm$ for all $m\in$ \emph{soc}$(M)\cap J(M)$.

\noindent(2) $r_{R}(m)\subseteq r_{R}(n)$, where $m\in$ \emph{soc}$(M)\cap J(M)$, $n\in M$ implies $Sn\subseteq Sm$.

\noindent(3) $l_{S}(mR\cap r_{M}(\alpha))=l_{S}(m)+S\alpha$, where $m\in $ \emph{soc}$(M)\cap J(M)$, $\alpha\in S$.

\noindent(4) If $kR$ is a simple submodule of $M$, then $Sk$ is a simple left $S$-module, for all $k\in J(M)$. Moreover, \emph{soc}$(M)\cap J(M)\subseteq $\,\,\emph{soc}$(_{S}M)$.

\noindent(5) \emph{soc}$(M)\cap J(M)\subseteq r_{M}(J($$_{S}S))$.

\noindent(6) $l_{S}(A\cap B)=l_{S}(A)+l_{S}(B)$, for every semisimple small right submodules $A$  and $B$  of $M$.
\end{lem}
\begin{proof}
(1) Let $n\in l_{M}r_{R}(m)$, thus $r_{R}(m)\subseteq r_{R}(n)$.
Now, let $\gamma:mR\longrightarrow M$ is given by $\gamma(mr)=nr$,
thus $\gamma$ is a well define $R$-homomorphism. By hypothesis, there exists an endomorphism $\beta$ of $M$
such that $\beta_{|mR}=\gamma$. Therefore, $n=\gamma(m)=\beta(m)\in Sm$,
that is $l_{M}r_{R}(m)\subseteq Sm$. The inverse inclusion is clear.

(2) Let $n\in M$ and $m\in$ soc$(M)\cap J(M)$. Since
$r_{R}(m)\subseteq r_{R}(n)$, then $n\in l_{M}r_{R}(m)$. By (1),
we have $n\in Sm$ as desired.

(3) If $f\in l_{S}(m)+S\alpha$, then $f=f_{1}+f_{2}$
such that $f_{1}(m)=0$ and $f_{2}=g\alpha$, for some $g\in S$.
For all $n\in mR\cap r_{M}(\alpha)$, we have $n=mr$ and $\alpha(n)=0$
for some $r\in R$. Since $f_{1}(n)=f_{1}(mr)=f_{1}(m)r=0$ and $f_{2}(n)=g(\alpha(n))=g(0)=0$,
thus $f\in l_{S}(mR\cap r_{M}(\alpha))$ and this implies that $l_{S}(m)+S\alpha\subseteq l_{S}(mR\cap r_{M}(\alpha))$.
Now, we will prove that the other inclusion. Let $g\in l_{S}(mR\cap r_{M}(\alpha)).$ If
$r\in r_{R}(\alpha(m))$, then $\alpha(mr)=0$, so $mr\in mR\cap r_{M}(\alpha)$
which yields $r_{R}(\alpha(m))\subseteq r_{R}(g(m))$. Since $m\in$ soc$(M)\cap J(M)$,
thus $\alpha(m)\in$ soc$(M)\cap J(M)$. By (2), we have that $g(m)=\gamma\alpha(m)$
for some $\gamma\in S$. Therefore, $g-\gamma\alpha\in l_{S}(m)$
which leads to $g\in l_{S}(m)+S\alpha$. Thus $l_{S}(mR\cap r_{M}(\alpha))=l_{S}(m)+S\alpha$.

(4) To prove $Sk$ is simple left $S$-module,
we need only show that $Sk$ is cyclic for any nonzero element in
it. If $0\neq\alpha(k)\in Sk$, then $\alpha:kR\longrightarrow\alpha(kR)$
is an $R$-isomorphism.
Since $\alpha\in S$, then $\alpha(kR)\ll M$. Since $M$
is ss-quasi-injective, thus $\alpha^{-1}:\alpha(kR)\longrightarrow kR$ has an extension
$\beta\in S$ and hence $\beta(\alpha(k))=\alpha^{-1}(\alpha(k))=k$,
so $k\in S\alpha k$ which leads to $Sk=S\alpha k$. Therefore $Sk$
is a simple left $S$-module and this leads to soc$(M)\cap J(M)\subseteq$ soc$(_{S}M)$.

(5) If $mR$ is simple and small submodule of $M$,
then $m\neq0$. We claim that $\alpha(m)=0$ for all $\alpha\in J(S)$,
thus $mR\subseteq r_{M}(J(S))$. Otherwise, $\alpha(m)\neq0$ for
some $\alpha\in J(S)$. Thus $\alpha:mR\longrightarrow\alpha(mR)$
is an $R$-isomorphism.
Now, we need prove that $r_{R}(\alpha(m))=r_{R}(m)$. Let $r\in r_{R}(m)$,
so $\alpha(m)r=\alpha(mr)=\alpha(0)=0$ which leads to $r_{R}(m)\subseteq r_{R}(\alpha(m))$.
The other inclusion, if $r\in r_{R}(\alpha(m))$, then $\alpha(mr)=0$,
that is $mr\in$ ker$(\alpha)=0$, so $r\in r_{R}(m)$. Hence $r_{R}(\alpha(m))=r_{R}(m)$.
Since $m,\alpha(m)\in$ soc$(M)\cap J(M)$, thus $S\alpha m=Sm$ (by(2))
and this implies that $m=\beta\alpha(m)$ for some $\beta\in S$,
so $(1-\beta\alpha)(m)=0$. Since $\alpha\in J(S)$, then the element
$\beta\alpha$ is quasi-regular by \cite[Theorem 15.3]{3AnFu74}. Thus $1-\beta\alpha$
is invertible and hence $m=0$ which is a contradiction. This shows
that soc$(M)\cap J(M)\subseteq r_{M}(J(S))$.

(6) Let $\alpha\in l_{S}(A\cap B)$ and consider
$f:A+B\longrightarrow M$ is given by $f(a+b)=\alpha(a),$ for all
$a\in A$ and $b\in B$. Since $M$ is ss-quasi-injective,
thus there exists $\beta\in S$ such that $f(a+b)=\beta(a+b).$ Thus
$\beta(a+b)=\alpha(a)$, so $(\alpha-\beta)(a)=\beta(b)$ which yields
$\alpha-\beta\in l_{S}(A)$. Therefore, $\alpha=\alpha-\beta+\beta\in l_{S}(A)+l_{S}(B)$
and this implies that $l_{S}(A\cap B)\subseteq l_{S}(A)+l_{S}(B)$.
The other inclusion is trivial and the proof is complete.
\end{proof}

\begin{rem}\label{Remark:(2.14)} \emph{Let $M$  be a right $R$-module, then $D(S)=\{
\alpha\in S=$ End$(M)\mid  r_{M}(\alpha)\cap mR\neq 0$ for each $
0\neq m\in $ soc$(M)\cap J(M)\}$ is a left ideal in $S$.}
\end{rem}
\begin{proof}
This is obvious.
\end{proof}

\begin{prop}\label{Proposition:(2.15)} Let $M$  be an ss-quasi-injective right $R$-module. Then $r_{M}(\alpha)\varsubsetneqq r_{M}(\alpha-\alpha\gamma\alpha)$, for all $\alpha\notin D(S)$ and for some $\gamma\in S$.
\end{prop}
\begin{proof}
For all $\alpha\notin D(S)$. By hypothesis, we can find $0\neq m\in$ soc$(M)\cap J(M)$ such that $r_{M}(\alpha)\cap mR=0$.
Clearly, $r_{R}(\alpha(m))=r_{R}(m)$, so $Sm=S\alpha m$ by Lemma~\ref{Lemma:(2.13)}(2). Thus $m=\gamma\alpha m$ for some $\gamma\in S$ and this
implies that $(\alpha-\alpha\gamma\alpha)m=0$. Therefore, $m\in r_{M}(\alpha-\alpha\gamma\alpha)$, but $m\notin r_{M}(\alpha)$ and hence the inclusion is strictly.
\end{proof}

\begin{prop}\label{Proposition:(2.16)}
Let $M$  be an ss-quasi-injective right $R$-module, then the set $\{\alpha\in S=\emph{End}(M)\mid 1-\beta\alpha$ is monomorphism for all
$\beta\in S\}$ is contained in $D(S)$. Moreover, $J($$_{S}S)\subseteq D(S)$.
\end{prop}
\begin{proof}
Let $\alpha\notin D(S)$, then there exists $0\neq m\in$ soc$(M)\cap J(M)$
such that $r_{M}(\alpha)\cap mR=0$. If $r\in r_{R}(\alpha(m))$,
then $\alpha(mr)=0$ and so $mr\in r_{M}(\alpha)$. Since $r_{M}(\alpha)\cap mR=0$.
Thus $r\in r_{R}(m)$ and hence $r_{R}(\alpha(m))\subseteq r_{R}(m)$,
so $Sm\subseteq S\alpha m$ by Lemma~\ref{Lemma:(2.13)}(2). Therefore, $m\in$ ker$(1-\gamma\alpha)$
for some $\gamma\in S$. Since $m\neq0$, thus $1-\gamma\alpha$ is
not monomorphism and hence the inclusion holds. Now, let $\alpha\in J($$_{S}S)$
we have $\beta\alpha$ is a quasi-regular element by \cite[Theorem 15.3]{3AnFu74}
and hence $1-\beta\alpha$ is isomorphism for all $\beta\in S$, which
completes the proof.
\end{proof}

\begin{thm}\label{Theorem:(2.17)}\textup{(ss-Baer's condition)}  The following statements are equivalent for a ring $R$.

\noindent (1) $M$ is an ss-injective right $R$-module.

\noindent (2) If $S_{r}\cap J=A\oplus B$ and $\alpha:A\longrightarrow M$
is an $R$-homomorphism, then there exists $m\in M$ such that $\alpha(a)=ma$ for all $a\in A$
and $mB=0$.

\noindent (3) If $S_{r}\cap J=A\oplus B$, and $\alpha:A\longrightarrow M$ is an $R$-homomorphism, then there exists $m\in M$ such that $\alpha(a)=ma$, for all $a\in A$
and $mB=0$.
\end{thm}
\begin{proof}
(1)$\Rightarrow$(2) Define $\gamma:S_{r}\cap J\longrightarrow M$ by $\gamma(a+b)=\alpha(a)$ for all $a\in A,b\in B$. By hypothesis,
there is a right $R$-homomorphism $\beta:\mathit{R\longrightarrow M}$ is an extension of $\gamma$, so if $\mathit{m=\beta}(1)$, then $\mathit{\alpha}(a)=\gamma(a)=\beta(a)=\beta(1)a=ma$, for all $\mathit{a\in A}$. Moreover, $mb=\beta(b)=\gamma(b)=\alpha(0)=0$
for all $b\in B$, so $mB=0$.

(2)$\Rightarrow$(1) Let $\alpha:I\rightarrow M$
be any right $R$-homomorphism, where $I$  is any semisimple small right ideal in $R$. By (2), there exists $m\in M$ such that $\alpha(a)=ma$
for all $a\in I$. Define $\beta:R_{R}\longrightarrow M$ by $\beta(r)=mr$
for all $r\in R$, thus $\beta$ extends $\alpha$.

(2)$\Leftrightarrow$(3) Clear.
\end{proof}

 A ring $R$ is called right universally mininjective ring if it satisfies the condition $\mathit{S}_{r}\cap J=0$ (see for example \cite{14NiYo97}). In the next results, we give new characterizations of universally mininjective ring in terms of ss-injectivity and soc-injectivity.

\begin{cor}\label{Corollary:(2.18)}  The following are equivalent for a ring $R$.

\noindent (1) $R$ is right universally mininjective.

\noindent (2) $R$  is right mininjective and every quotient of a soc-injective right
$R$-module is soc-injective.

\noindent (3) $R$  is right mininjective and every quotient of an injective right $R$-module is soc-injective.

\noindent (4) $R$  is right mininjective and every semisimple submodule of a projective
right $R$-module is projective.

\noindent (5) Every right $R$-module is ss-injective.

\noindent (6) Every simple   right ideal is ss-injective.

\end{cor}
\begin{proof}
(1)$\Leftrightarrow$(2)$\Leftrightarrow$(3)$\Leftrightarrow$(4) By \cite[Lemma 5.1]{14NiYo97} and \cite[Corollary 2.9]{2AmYoZe05}.

(1)$\Leftrightarrow$(5)$\Leftrightarrow$(6) By Proposition~\ref{Proposition:(2.11)}.
\end{proof}

\begin{thm}\label{Theorem:(2.20)} If $M$  is a projective right $R$-module. Then the following statements are equivalent.

\noindent (1) Every quotient of an ss-$M$-injective right $R$-module is ss-$M$-injective.

\noindent (2) Every quotient of a \emph{soc}-$M$-injective right $R$-module is ss-$M$-injective.

\noindent (3) Every quotient of an injective right $R$-module is ss-$M$-injective.

\noindent (4) Every sum of two ss-$M$-injective submodules of a right $R$-module is ss-$M$-injective.

\noindent (5) Every sum of two \emph{soc}-$M$-injective submodules of
a right $R$-module is ss-$M$-injective.

\noindent (6) Every sum of two injective submodules of a right $R$-module
is ss-$M$-injective.

\noindent (7) Every semisimple small submodule of $M$ is projective.

\noindent (8) Every simple small submodule of $M$ is projective.

\noindent (9) $\emph{soc}(M)\cap J(M)$ is projective.
\end{thm}
\begin{proof} (1)$\Rightarrow$(2)$\Rightarrow$(3), (4)$\Rightarrow$(5)$\Rightarrow$(6)
and (9)$\Rightarrow$(7)$\Rightarrow$(8) are obvious.

(8)$\Rightarrow$(9) Since soc$(M)\cap J(M)$ is a direct sum of simple submodules of $M$ and since every simple in $J(M)$ is small in $M$, thus soc$(M)\cap J(M)$ is projective.

(3)$\Rightarrow$(7) Consider the following diagram:
\[
\xymatrix{
0\,\ar[r] &K \,\ar[d]_{f}\ar@{^{(}->}[r]^{i} \ar[r] & M  \\
E \,\ar[r]^{h}&N \, \ar[r] &0
}
\]

\noindent where $E$  and $N$  are right $R$-modules, $K$  is a semisimple small submodule of $M$, $h$  is a right $R$-epimorphism
and \, $f$  \, is a right$R$-homomorphism. We can assume that $E$  is injective (see, e.g. \cite[Proposition 5.2.10]{6Bla11}). Since $N$
is ss-$M$-injective, thus $f$  can be extended to an $R$-homomorphism $\mathit{g}:M\longrightarrow N$. By projectivity of $M$,
thus $g$ can be lifted to an $R$-homomorphism $\tilde{g}:M\longrightarrow E$ such that $\mathit{h}\circ\tilde{g}=g$.
Define $\tilde{f}:K\longrightarrow E$ is the restriction of $\tilde{g}$
over $K$. Clearly, $\mathit{h}\circ\tilde{f}=f$ and this implies that $K$  is projective.

(7)$\Rightarrow$(1) Let $N$  and $L$  be right $R$-modules with $h:N\longrightarrow L$ is an $R$-epimorphism and $N$  is ss-$M$-injective. Let   $K$ be any  semisimple small submodule of $M$ and let $f:K\longrightarrow L$ be any left $R$-homomorphism. By hypothesis $K$  is projective, thus $f$  can be lifted to $R$-homomorphism $\mathit{g}:K\longrightarrow N$ such that $\mathit{h}\circ g=f$. Since $N$  is ss-$M$-injective, thus there exists an $R$-homomorphism
$\tilde{g}:M\longrightarrow N$ such that $\tilde{g}\circ i=g$. Put $\beta=h\circ\tilde{g}:M\longrightarrow L$. Thus $\beta\circ i=h\circ\tilde{g}\circ i=h\circ g=f$. Hence $L$ is an ss-$M$-injective right $R$-module.

(1)$\Rightarrow$(4) Let $N_{1}$ and $N_{2}$ be two ss-$M$-injective submodules of a right $R$-module
$N$. Thus $N_{1}+N_{2}$ is a homomorphic image of the direct sum $N_{1}\oplus N_{2}$. Since $N_{1}\oplus N_{2}$ is ss-$M$-injective, thus $N_{1}+N_{2}$ is ss-$M$-injective by hypothesis.

(6)$\Rightarrow$(3) Let $E$  be an injective right $R$-module with submodule $N$. Let $Q=E\oplus E$, \, $K=\{(n,n)\mid n\in N\}$, $\bar{Q}=Q/K$, $H_{1}=\{y+K\in\bar{Q}\mid  y\in E\oplus 0\}$, $H_{2}=\{y+K\in\bar{Q}\mid  y\in 0\oplus E\}$. Then $\bar{Q}=H_{1}+H_{2}$. Since $(E\oplus 0)\cap K=0$ and $(0\oplus E)\cap K=0$, thus $E\cong H_{i}$, $i=1,2$. Since $H_{1}\cap H_{2}=\{y+K\in\bar{Q}\mid y\in N\oplus0\}$=
$\{y+K\in\bar{Q}\mid y\in 0\oplus N\}$, thus $H_{1}\cap H_{2}\cong N$ under $y\mapsto y+K$ for all $y\in N\oplus0$.
By hypothesis, $\bar{Q}$ is ss-$M$-injective. Since $H_{1}$ is injective, thus $\bar{Q}=H_{1}\oplus A$ for some submodule $A$ of $\bar{Q}$, so $A\cong(H_{1}+H_{2})/H_{1}\cong H_{2}/H_{1}\cap H_{2}\cong E/N$. By Theorem~\ref{Theorem:(2.3)}(5), $E/N$ is ss-$M$-injective.
\end{proof}

\begin{cor}\label{Corollary:(2.21)} The following statements are equivalent.

\noindent (1) Every quotient of an ss-injective right $R$-module is ss-injective.

\noindent (2) Every quotient of a soc-injective right $R$-module is ss-injective.

\noindent (3) Every quotient of a small injective right $R$-module is ss-injective.

\noindent (4) Every quotient of an injective right $R$-module is ss-injective.

\noindent (5) Every sum of two ss-injective submodules of any right $R$-module is ss-injective.

\noindent (6) Every sum of two soc-injective submodules of any right $R$-module is ss-injective.

\noindent (7) Every sum of two small injective submodules of any right $R$-module is ss-injective.

\noindent (8) Every sum of two injective submodules of any right $R$-module is ss-injective.

\noindent (9) Every semisimple small submodule of  any projective right $R$-module  is projective.

\noindent (10) Every semisimple small submodule of any finitely generated projective right $R$-module  is projective.

\noindent (11) Every semisimple small submodule of  $\mathit{R}_{R}$  is projective.

\noindent (12) Every simple small submodule of   $\mathit{R}_{R}$ is projective.

\noindent (13)  $\mathit{S}_{r}\cap J$  is
projective.

\noindent (14) $S_{r}$  is projective.
\end{cor}
\begin{proof}
The equivalence of (1), (2), (4), (5), (6), (8), (11), (12) and (13) is from Theorem~\ref{Theorem:(2.20)}.

(1)$\Rightarrow$(3)$\Rightarrow$(4), (5)$\Rightarrow$(7)$\Rightarrow$(8) and (9)$\Rightarrow$(10)$\Rightarrow$(13) are clear.

(14)${\color{black}\Rightarrow}$(9) By \cite[Corollary 2.9]{2AmYoZe05}.

(13)$\Rightarrow$(14) Let  $S_{r}=(S_{r}\cap J)\oplus A$, where $A=\underset{{\scriptscriptstyle i\in I}}{{\displaystyle {\scriptstyle {\textstyle \bigoplus}}}}S_{i}$ and $S_{i}$ is a right simple and summand of $R_{R}$ for all $i\in I$. Thus $A$ is projective, but $S_{r}\cap J$ is projective, so it follows
that $S_{r}$ is projective.
\end{proof}

\begin{thm}\label{Theorem:(2.22)}
If every simple singular right $R$-module is ss-injective, then $r(a)\subseteq^{\oplus}R_{R}$ for every $a\in S_{r}\cap J$ and $S_{r}$ is projective.
\end{thm}
\begin{proof} Let $a\in S_{r}\cap J$ and let $A=RaR+r(a)$. Thus there exists a right ideal $B$ of $R$ such that $A\oplus B\subseteq^{ess}R_{R}$. Suppose that $A\oplus B\neq R_{R}$, thus we choose $I\subseteq^{max}R_{R}$ such that $A\oplus B\subseteq I$ and so $I\subseteq^{ess}R_{R}$. By hypothesis, $R/I$ is a right ss-injective. Consider the map $\alpha:aR\longrightarrow R/I$ is given by $\alpha(ar)=r+I$ which is a well-define $R$-homomorphism. Thus there exists $c\in R$ such that $1+I=ca+I$ and hence $1-ca\in I$. But $ca\in RaR\subseteq I$ which leads to $1\in I$, a contradiction.
Thus $A\oplus B=R$ and hence $RaR+(r(a)\oplus B)=R$. Since $RaR\ll R_{R}$, thus $r(a)\subseteq^{\oplus}R_{R}$. Put $r(a)=(1-e)R$, for some
$e^{2}=e\in R$, so it follows that $ax=aex$ \,\,for all $x\in R$ and hence $aR=aeR$. Let $\gamma:eR\longrightarrow aeR$ be defined by $\gamma(er)=aer$ for all $r\in R$. Then $\gamma$ is a well-defined $R$-epimorphism. Clearly, ker$(\gamma)=eR\cap r(a)$. Hence $\gamma$ is an isomorphism and so $aR$ is projective. Since $S_{r}\cap J$ is a direct sum of simple small right ideals, thus $S_{r}\cap J$ is projective and it follows from Corollary~\ref{Corollary:(2.21)}   that $S_{r}$ is projective.
\end{proof}

\begin{cor}\label{Corollary:(2.23)} The following statements are equivalent for a ring $R$.

\noindent  (1) $R$ is right mininjective and every simple singular right $R$-module is ss-injective.

\noindent  (2) $R$ is right universally mininjective.
\end{cor}
\begin{proof} By Theorem~\ref{Theorem:(2.22)}   and \cite[Lemma 5.1]{14NiYo97}.
\end{proof}

Recall that a ring $R$  is called zero insertive, if $aRb=0$ for each $a,b\in R$ with $ab=0$ (see \cite{19ThQu09}). Note that if $R$
is zero insertive ring, then $RaR+r(a)\subseteq^{ess}R_{R}$ for every $a\in R$ (see \cite[Lemma 2.11]{19ThQu09}).

\begin{prop}\label{Proposition:(2.24)}   Let \, $R$ \, be   a zero insertive ring.  If every   simple singular right \, $R$-module is ss-injective, then $R$ is right universally mininjective.
\end{prop}
\begin{proof} Let $a\in S_{r}\cap J$. We claim that $RaR+r(a)=R$, thus $r(a)=R$ (since $RaR\ll R$), so $a=0$ and this means that
$S_{r}\cap J=0$. Otherwise, if $RaR+r(a)\subsetneqq R$, then there exists a maximal right ideal $I$  of $R$ such that $RaR+r(a)\subseteq I$. Since $I\subseteq^{ess}R_{R}$, thus $R/I$ is ss-injective by hypothesis. Consider $\alpha:aR\longrightarrow R/I$ is given by $\alpha(ar)=r+I$ \, for all $r\in R$ which is a well-defined $R$-homomorphism. Thus $1+I=ca+I$ for some $c\in R$. Since $ca\in RaR\subseteq I$, thus $1\in I$ and this a contradicts with a maximality of $I$, so we must have $RaR+r(a)=R$ and this completes the proof.
\end{proof}

\begin{thm}\label{Theorem:(2.25)}
If $M$  is a finitely generated right $R$-module, then the following statements are equivalent.

\noindent  (1) \emph{soc}$(M)\cap J(M)$ is a Noetherian $R$-module.

\noindent(2) \emph{soc}$(M)\cap J(M)$ is finitely generated.

\noindent (3) Any direct sum of ss-$M$-injective right $R$-modules is ss-$M$-injective.

\noindent (4) Any direct sum of \emph{soc}-$M$-injective right $R$-modules is ss-$M$-injective.

\noindent(5) Any direct sum of injective right $R$-modules is ss-$M$-injective.

\noindent (6)  $\mathit{K}^{(S)}$ is ss-$M$-injective for every injective right $R$-module
$K$  and for any index set $S$.

\noindent (7) $\mathit{K}^{(\mathbb{N})}$ is ss-$M$-injective for every injective right $R$-module $K$.
\end{thm}
\begin{proof}
\textcolor{black}{(1)$\Rightarrow$(2) and (3)$\Rightarrow$(4)$\Rightarrow$(5)$\Rightarrow$(6)$\Rightarrow$(7)
Clear.}

\textcolor{black}{(2)$\Rightarrow$(3) Let $\mathit{E}=\underset{{\scriptscriptstyle i\in I}}{\bigoplus}M_{i}$
be a direct sum of ss-}\textit{\textcolor{black}{M}}\textcolor{black}{-injective
right }\textit{\textcolor{black}{R}}\textcolor{black}{-modules and
$\mathit{f}:N\longrightarrow E$ be a right }\textit{\textcolor{black}{R}}\textcolor{black}{-homomorphism,
where }\textit{\textcolor{black}{N}}\textcolor{black}{{} is a semisimple
small submodule of }\textit{\textcolor{black}{M}}\textcolor{black}{.
Since soc$(M)\cap J(M)$ is finitely generated, thus}\textit{\textcolor{black}{{}
N}}\textcolor{black}{{} is finitely generated and hence $\mathit{f}(N)\subseteq\underset{{\scriptscriptstyle j\in I_{1}}}{\bigoplus}M_{j}$,
for some finite subset }\textit{\textcolor{black}{$I_{1}$}}\textcolor{black}{{}
of }\textit{\textcolor{black}{I}}\textcolor{black}{. Since a finite
direct sums of ss-}\textit{\textcolor{black}{M}}\textcolor{black}{-injective
right }\textit{\textcolor{black}{R}}\textcolor{black}{-modules is
ss-}\textit{\textcolor{black}{M}}\textcolor{black}{-injective, thus
$\underset{{\scriptscriptstyle j\in I_{1}}}{\bigoplus}M_{j}$ is ss-}\textit{\textcolor{black}{M}}\textcolor{black}{-injective
and hence }\textit{\textcolor{black}{f}}\textcolor{black}{{} can be
extended to an }\textit{\textcolor{black}{R}}\textcolor{black}{-homomorphism
$\mathit{g}:M\longrightarrow E$. Thus }\textit{\textcolor{black}{E}}\textcolor{black}{{}
is ss-}\textit{\textcolor{black}{M}}\textcolor{black}{-injective. }

\textcolor{black}{(7)$\Rightarrow$(1) Let $\mathit{N}_{1}\subseteq N_{2}\subseteq...$
be a chain of submodules of soc$(M)\cap J(M)$. For each $\mathit{i}\geq1$,
let $\mathit{E}_{i}=E(M/N_{i})$,   $\mathit{E}=\underset{{\scriptscriptstyle i=1}}{\overset{{\scriptscriptstyle \infty}}{\bigoplus}}E_{i}$ and $\mathit{M_{i}}=\underset{{\scriptscriptstyle j=1}}{\overset{{\scriptscriptstyle \infty}}{\prod}}E_{j}=E_{i}\oplus(\underset{{\scriptscriptstyle \underset{j\neq i}{j=1}}}{\overset{{\scriptscriptstyle \infty}}{\prod}}E_{j})$,
then $\mathit{M_{i}}$ is injective. By hypothesis, $\mathit{\underset{{\scriptscriptstyle i=1}}{\overset{{\scriptscriptstyle \infty}}{\bigoplus}}M_{i}}=(\underset{{\scriptscriptstyle i=1}}{\overset{{\scriptscriptstyle \infty}}{\bigoplus}}E_{i})\oplus(\underset{{\scriptscriptstyle i=1}}{\overset{{\scriptscriptstyle \infty}}{\bigoplus}}\underset{{\scriptscriptstyle \underset{j\neq i}{j=1}}}{\overset{{\scriptscriptstyle \infty}}{\prod}}E_{j})$
is ss-}\textit{\textcolor{black}{M}}\textcolor{black}{-injective, so
it follows from Theorem~\ref{Theorem:(2.3)}(5)   that }\textit{\textcolor{black}{E}}\textcolor{black}{{}
it self is ss-}\textit{\textcolor{black}{M}}\textcolor{black}{-injective.
Define $\mathit{f}:U=\overset{{\scriptscriptstyle \infty}}{\underset{{\scriptscriptstyle i=1}}{\bigcup}}N_{i}\longrightarrow E$
by $\mathit{f}(m)=(m+N_{i})_{i}$. It is clear that }\textit{\textcolor{black}{f}}\textcolor{black}{{}
is a well defined }\textit{\textcolor{black}{R}}\textcolor{black}{-homomorphism. Since
}\textit{\textcolor{black}{M}}\textcolor{black}{{} is finitely generated,
thus soc$(M)\cap J(M)$ is a semisimple small submodule of
}\textit{\textcolor{black}{M}}\textcolor{black}{{} and hence $\overset{{\scriptscriptstyle \infty}}{\underset{{\scriptscriptstyle i=1}}{\bigcup}}N_{i}$
is a semisimple small submodule of }\textit{\textcolor{black}{M}}\textcolor{black}{,
so }\textit{\textcolor{black}{f}}\textcolor{black}{{} can be extended
to a right }\textit{\textcolor{black}{R}}\textcolor{black}{-homomorphism
$\mathit{g}:M\longrightarrow E$. Since }\textit{\textcolor{black}{M}}\textcolor{black}{{}
is finitely generated, we have $\mathit{g}(M)\subseteq\overset{{\scriptscriptstyle n}}{\underset{{\scriptscriptstyle i=1}}{\bigoplus}}E(M/N_{i})$
for some }\textit{\textcolor{black}{n}}\textcolor{black}{{} and hence
$\mathit{f}(\underset{{\scriptscriptstyle i=1}}{\overset{{\scriptscriptstyle \infty}}{\bigcup}}N_{i})\subseteq\underset{{\scriptscriptstyle i=1}}{\overset{{\scriptscriptstyle n}}{\bigoplus}}E(M/N_{i})$.
Since $\mathit{\pi_{i}}f(x)=\pi_{i}(x+N_{j})_{{\scriptscriptstyle j\geq1}}=x+N_{i}$,
for all $\mathit{x\in}U$ and $\mathit{i}\geq1$, where $\mathit{\pi_{i}}:\underset{{\scriptscriptstyle j\geq1}}{\bigoplus}E(M/N_{j})\longrightarrow E(M/N_{i})$
be the projection map, thus $\mathit{\pi_{i}}f(U)=U/N_{i}$ for all
$\mathit{i}\geq1$. Since $\mathit{f}(U)\subseteq\underset{{\scriptscriptstyle i=1}}{\overset{{\scriptscriptstyle n}}{\bigoplus}}E(M/N_{i})$, thus $\mathit{U}/N_{i}=\pi_{i}f(U)=0$, for all $\mathit{i}\geq n+1$,
so $\mathit{U}=N_{i}$ for all $\mathit{i}\geq n+1$ and hence the
chain $\mathit{N}_{1}\subseteq N_{2}\subseteq...$ terminates at $\mathit{N}_{n+1}$.
Thus soc$(M)\cap J(M)$ is a  Noetherian $R$-module.}
\end{proof}

\begin{cor}\label{Corollary:(2.26)} If $N$ is a finitely generated right $R$-module, then the following statements are equivalent.

\noindent (1)   \emph{soc}$(N)\cap J(N)$ is finitely generated.

\noindent (2)  $\mathit{M}^{(S)}$ is ss-$N$-injective for every \emph{soc}-$N$-injective right $R$-module $M$
and for any index set $S$.

\noindent (3)  $\mathit{M}^{(S)}$ is ss-$N$-injective for every ss-$N$-injective right $R$-module $M$
and for any index set $S$.

\noindent (4)  $\mathit{M}^{(\mathbb{N})}$
is ss-$N$-injective for every \emph{soc}-$N$-injective right $R$-module $M$.

\noindent (5)  $\mathit{M}^{(\mathbb{N})}$ is ss-$N$-injective for
every ss-$N$-injective
right $R$-module $M$.
\end{cor}
\begin{proof}
By Theorem~\ref{Theorem:(2.25)}.
\end{proof}

\begin{cor}\label{Corollary:(2.27)} The following statements are equivalent.

\noindent (1)   $\mathit{S}_{r}\cap J$ is  finitely generated.

\noindent (2) Any direct sum of ss-injective right $R$-modules is ss-injective.

\noindent (3) Any direct sum of soc-injective right $R$-modules is ss-injective.

\noindent (4) Any direct sum of small injective right $R$-modules is ss-injective.

\noindent (5) Any direct sum of injective right $R$-modules is ss-injective.

\noindent (6)  $\mathit{M}^{(S)}$ is ss-injective for every injective right $R$-module $M$  and for any index set $S$.

\noindent (7)  $\mathit{M}^{(S)}$ is ss-injective for every soc-injective right $R$-module $M$  and for any index set $S$.

\noindent (8) $\mathit{M}^{(S)}$ is ss-injective for every small injective right $R$-module $M$ and for any index set $S$.

\noindent (9)  $\mathit{M}^{(S)}$ is ss-injective for every ss-injective right $R$-module $M$ and for any index set $S$.

\noindent (10) $\mathit{M}^{(\mathbb{N})}$ is ss-injective for every injective right $R$-module $M$.

\noindent (11)  $\mathit{M}^{(\mathbb{N})}$ is ss-injective for every \emph{soc}-injective right  $R$-module $M$.

\noindent (12)  $\mathit{M}^{(\mathbb{N})}$ is ss-injective for every small injective right $R$-module $M$.

\noindent (13)  $\mathit{M}^{(\mathbb{N})}$ is ss-injective for every ss-injective right $R$-module $M$.
\end{cor}
\begin{proof}
By applying Theorem~\ref{Theorem:(2.25)} and Corollary~\ref{Corollary:(2.26)}.
\end{proof}

\begin{rem}\label{Remark:(2.28)} \emph{Let $M$   be a right $R$-module. We denote that $r_{u}(N)=\{a\in S_{r}\cap J \mid Na=0\}$ and $l_{M}(K)=\{
m\in M \mid mK=0\}$ where $N\subseteq M$ and $K\subseteq S_{r}\cap J$. Clearly, $r_{u}(N)\subseteq(S_{r}\cap J)_{R}$
and $l_{M}(K)\subseteq$ $ _{s}{M}$, where $S=End(M_{R})$ and we have the
following: }

\noindent\textcolor{black}{(1) $N\subseteq l_{M}r_{u}(N)$ \emph{for all} $N\subseteq M$.}

\noindent\textcolor{black}{(2) $K\subseteq r_{u}l_{M}(K)$ \emph{for all} $K\subseteq S_{r}\cap J$.}

\noindent\textcolor{black}{(3) $r_{u}l_{M}r_{u}(N)=r_{u}(N)$ \emph{for all} $N\subseteq M$.}

\noindent\textcolor{black}{(4) $l_{M}r_{u}l_{M}(K)=l_{M}(K)$ \emph{for all} $K\subseteq S_{r}\cap J$.}
\end{rem}
\begin{proof}
This is clear
\end{proof}

\begin{lem}\label{Lemma:(2.29)} The following statements are equivalent for a right $R$-module $M$:

\noindent (1) $R$  satisfies the ACC for right ideals of form $r_{u}(N)$, where $N\subseteq M$.

\noindent (2) $R$  satisfies the DCC for $l_{M}(K)$, where $K\subseteq S_{r}\cap J$.

\noindent (3) For each semisimple small right ideal $I$ there exists a finitely
generated right ideal $K\subseteq I$ such that $l_{M}(I)=l_{M}(K)$.
\end{lem}
\begin{proof} (1)$\Leftrightarrow$(2) Clear.

(2)$\Rightarrow$(3) Consider $\Omega=\{
l_{M}(A) \mid A$ is finitely generated right ideal and $A\subseteq I$ $\}$ which
is non empty set because $M\in\Omega$. Now, let $K$
be a finitely generated right ideal of $R$ and contained in $I$.
such that $l_{M}(K)$ is minimal in $\Omega$. Put $B=K+xR$, where
$x\in I$. Thus $B$ is a finitely generated right ideal contained
in $I$ and $l_{M}(B)\subseteq l_{M}(K)$. But since $l_{M}(K)$ is
minimal in $\Omega$, thus $l_{M}(B)=l_{M}(K)$ which yields $l_{M}(K)x=0$
for all $x\in I$. Therefore, $l_{M}(K)I=0$ and hence $l_{M}(K)\subseteq l_{M}(I)$.
But $l_{M}(I)\subseteq l_{M}(K)$, so $l_{M}(I)=l_{M}(K)$.

(3)$\Rightarrow$(1) Suppose that $r_{u}(M_{1})\subseteq r_{u}(M_{2})\subseteq...\subseteq r_{u}(M_{n})\subseteq...$,
where $M_{i}\subseteq M$ for each $i$. Put $D_{i}=l_{M}r_{u}(M_{i})$ for each $i$,
and $I=\underset{{\scriptscriptstyle i=1}}{\overset{{\scriptscriptstyle \infty}}{\bigcup}}r_{u}(M_{i})$,
then $I\subseteq S_{r}\cap J$. By hypothesis, there exists a finitely
generated right ideal $K$ of $R$ and contained in $I$ such that $l_{M}(I)=l_{M}(K)$.
Since $K$  is a finitely generated, thus there exists $t\in\mathbb{N}$ such that $K\subseteq r_{u}(M_{n})$
for all $n\geq t$, that is $l_{M}(K)\supseteq l_{M}r_{u}(M_{n})=D_{n}$
for all $n\geq t$. Since $l_{M}(K)=l_{M}(I)=l_{M}(\underset{{\scriptscriptstyle i=1}}{\overset{{\scriptscriptstyle \infty}}{\bigcup}}r_{u}(M_{i}))=\overset{{\scriptscriptstyle \infty}}{\underset{{\scriptscriptstyle i=1}}{\bigcap}}l_{M}r_{u}(M_{i})=\overset{{\scriptscriptstyle \infty}}{\underset{{\scriptscriptstyle i=1}}{\bigcap}}D_{i}\subseteq D_{n}$,
thus $l_{M}(K)=D_{n}$ for all $n\geq t$. Since $D_{n}=l_{M}r_{u}(M_{n})$,
thus $r_{u}(M_{n})=r_{u}l_{M}r_{u}(M_{n})=r_{u}(D_{n})=r_{u}l_{M}(K)$
for all $n\geq t$. Thus $r_{u}(M_{n})=r_{u}(M_{t})$ for all $n\geq t$.
and hence (3) implies (1), which completes the proof.
\end{proof}

 The first part in following proposition is obtained directly by Corollary~\ref{Corollary:(2.27)}, but we will prove it by different way.

\begin{prop}\label{Proposition:(2.30)} Let $E$ be an ss-injective right $R$-module. Then $E^{(\mathbb{N})}$ is ss-injective if and only if $R$
satisfies the ACC  for right ideals of form $r_{u}(N)$, where $N\subseteq E$.
\end{prop}
\begin{proof} ($\Rightarrow$) Suppose that $r_{u}(N_{1})\subsetneqq r_{u}(N_{2})\subsetneqq...\subsetneqq r_{u}(N_{m})\subsetneqq...$
be a strictly chain, where $N_{i}\subseteq E$. Thus we get, $l_{E}r_{u}(N_{1})\supsetneqq l_{E}r_{u}(N_{2})\supsetneqq...\supsetneqq l_{E}r_{u}(N_{m})\supsetneqq...$. For each $i\geq1,$ so we can find $t_{i}\in l_{E}r_{u}(N_{i})/l_{E}r_{u}(N_{i+1})$
and $a_{i+1}\in r_{u}(N_{i+1})$ such that $t_{i}a_{i+1}\neq0$. Let $L=\overset{{\scriptscriptstyle \infty}}{\underset{{\scriptscriptstyle i=1}}{\bigcup}}r_{u}(N_{i})$, then for all $\ell\in L$ there exists $m_{\ell}\geq1$ such that
$\ell\in r_{u}(N_{i})$ for all $i\geq m_{\ell}$ and this implies that $t_{i}\ell=0$ for all $i\geq m_{\ell}$. Put $\bar{t}=(t_{i})_{i}$
, we have $\bar{t}\ell\in E^{(\mathbb{N})}$ for every $\ell\in L$. Consider $\alpha_{\bar{t}}:L\longrightarrow E^{(\mathbb{N})}$ is
given by $\alpha_{\bar{t}}(\ell)=\bar{t}\ell$, then $\alpha_{\bar{t}}$ is a well-define $R$-homomorphism.
Since $L$  is a semisimple small right ideal, thus $\alpha_{\bar{t}}$ extends to $\gamma:R\longrightarrow E^{(\mathbb{N})}$
(by hypothesis) and hence $\alpha_{\bar{t}}(\ell)=\bar{t}\ell=\gamma(\ell)=\gamma(1)\ell$.
Thus there exists $k\geq1$ such that $t_{i}\ell=0$ for all $i\geq k$
and all $\ell\in L$ (since $\gamma(1)\in E^{(\mathbb{N})}$), but
this contradicts with $t_{k}a_{k+1}\neq 0$.

($\Leftarrow$) Let $\alpha:I\longrightarrow E^{(\mathbb{N})}$
be an $R$-homomorphism, where $I$  is a semisimple small right ideal, thus it follows from Lemma~\ref{Lemma:(2.29)}   that there is a finitely generated right ideal $K\subseteq I$ such that $l_{M}(I)=l_{M}(K)$. Since $E^{\mathbb{N}}$ is ss-injective,
thus $\alpha=a.$ for some $a\in E^{\mathbb{N}}$. Write $K=\overset{{\scriptscriptstyle m}}{\underset{{\scriptscriptstyle i=1}}{\bigoplus}}r_{i}R$,
so we have $\alpha(r_{i})=ar_{i}\in E^{(\mathbb{N})}$, $i=1,2,...,m$.
Thus there exists $\tilde{a}\in E^{(\mathbb{N})}$ such that $a_{n}r_{i}=\tilde{a}_{n}r_{i}$
for all $n\in\mathbb{N}$, $i=1,2,...,m,$ where $a_{n}$ is the $n$th-coordinate of $a$ . Since $K$ is generated by $\{r_{1},r_{2},...,r_{m}\}$, thus $ar=\tilde{a}r$ for all $r\in K$. Therefore, $a_{n}-\tilde{a}_{n}\in l_{M}(K)=l_{M}(I)$ for all $n\in\mathbb{N}$ which leads to $a_{n}r=\tilde{a}_{n}r$
for all $r\in I$ and $n\in\mathbb{N}$, so $ar=\tilde{ar}$ for all $r\in I$. Thus there exists $\tilde{a}\in E^{(\mathbb{N})}$ such that $\alpha(r)=\tilde{a}r$ for all $r\in I$ and this means that $E^{(\mathbb{N})}$ is ss-injective.
\end{proof}

\begin{thm}\label{Theorem:(2.31)} The following statements are equivalent for a ring $R$:

\noindent (1) $S_{r}\cap J$ is finitely generated.

\noindent  (2) $\overset{{\scriptscriptstyle \infty}}{\underset{{\scriptscriptstyle i=1}}{\bigoplus}}E(M_{i})$
is ss-injective right $R$-module for every simple right $R$-modules $M_{i}$, $i\geq1$.
\end{thm}
\begin{proof} (1)$\Rightarrow$(2) By Corollary~\ref{Corollary:(2.27)}.

(2)$\Rightarrow$(1) Let $I_{1}\subsetneqq I_{2}\subsetneqq...$
be a properly ascending chain of semisimple small right ideals of
$R$. Clearly, $I=\overset{{\scriptscriptstyle \infty}}{\underset{{\scriptstyle {\scriptscriptstyle i=1}}}{\bigcup}}I_{i}\subseteq S_{r}\cap J$.
For every $i\geq1$, there exists $a_{i}\in I$, $a_{i}\notin I_{i}$
and consider $N_{i}/I_{i}\subseteq^{max}(a_{i}R+I_{i})/I_{i}$, so
$K_{i}=(a_{i}R+I_{i})/N_{i}$ is a simple right $R$-module.
Define $\alpha_{i}:(a_{i}R+I_{i})/I_{i}\longrightarrow(a_{i}R+I_{i})/N_{i}$
by $\alpha_{i}(x+I_{i})=x+N_{i}$ which is right $R$-epimorphism.
Let $E(K_{i})$ be the injective hull of $K_{i}$ and $i_{i}:K_{i}\rightarrow E(K_{i})$ be the inclusion map. By injectivity of $E(K_{i})$,  there there exists $\beta_{i}:I/I_{i}\longrightarrow E(K_{i})$
such that $\beta_{i}=i_{i}\alpha_{i}$. Since $a_{i}\notin N_{i}$,
then $\beta_{i}(a_{i}+I_{i})=i_{i}(\alpha_{i}(a_{i}+I_{i}))=a_{i}+N_{i}\neq0$
for each $i\geq1$. If $b\in I$, then there exists $n_{b}\geq1$
such that $b\in I_{i}$ for all $i\geq n_{b}$ and hence $\beta_{i}(b+I_{i})=0$
for all $i\geq n_{b}$. Thus we can define $\gamma:I\longrightarrow\overset{{\scriptscriptstyle \infty}}{\underset{{\scriptscriptstyle i=1}}{\bigoplus}}E(K_{i})$
by $\gamma(b)=(\beta_{i}(b+I_{i}))_{i}$. Then there exists $\tilde{\gamma}:R\longrightarrow\overset{{\scriptscriptstyle \infty}}{\underset{{\scriptscriptstyle i=1}}{\bigoplus}}E(K_{i})$
such that $\tilde{\gamma}_{|I}=\gamma$ (by hypothesis). Put $\tilde{\gamma}(1)=(c_{i})_{i}$,
thus there exists $n\geq1$ with $c_{i}=0$ for all $i\geq n$. Since
$(\beta_{i}(b+I_{i}))_{i}=\gamma(b)=\tilde{\gamma}(b)=\tilde{\gamma}(1)b=(c_{i}b)_{i}$
for all $b\in I$, thus $\beta_{i}(b+I_{i})=c_{i}b$ for all $i\geq1$,
so it follows that $\beta_{i}(b+I_{i})=0$ for all $i\geq n$ and
all $b\in I$ and this contradicts with $\beta_{n}(a_{n}+I_{n})\neq0$.
Hence (2) implies (1).
\end{proof}

\section{Strongly SS-Injective Modules}

\begin{prop}\label{Proposition:(3.1)} The following statements are equivalent.

\noindent (1) $M$ is a strongly ss-injective right$R$-module.

\noindent (2) Every $R$-homomorphism $\alpha:A\longrightarrow M$ extends to $N$,
for all right $R$-module $N$, where $A\ll N$ and $\alpha(A)$ is a semisimple submodule in  $M$.\end{prop}
\begin{proof}
\textcolor{black}{(2)$\Rightarrow$(1) Clear.}

\textcolor{black}{(1)$\Rightarrow$(2) Let }\textit{\textcolor{black}{A}}\textcolor{black}{{}
be a small  submodule of }\textit{\textcolor{black}{N}}\textcolor{black}{,
and $\alpha:A\longrightarrow M$ be an }\textit{\textcolor{black}{R}}\textcolor{black}{-homomorphism
with $\alpha(A)$ is a semisimple submodule of }\textit{\textcolor{black}{M}}\textcolor{black}{.
If $B=$ ker$(\alpha)$, then $\alpha$ induces an embedding $\tilde{\alpha}:A/B\longrightarrow M$
defined by $\tilde{\alpha}(a+B)=\alpha(a)$, for all $a\in A$. Clearly,
$\tilde{\alpha}$ is well define because if $a_{1}+B=a_{2}+B$ we
have $a_{1}-a_{2}\in B$, so $\alpha(a_{1})=\alpha(a_{2})$, that
is $\tilde{\alpha}(a_{1}+B)=\tilde{\alpha}(a_{2}+B)$. Since }\textit{\textcolor{black}{M}}\textcolor{black}{{}
is strongly ss-injective and $A/B$ is semisimple and small in $N/B$,
thus $\tilde{\alpha}$ extends to an }\textit{\textcolor{black}{R}}\textcolor{black}{-homomorphism
$\gamma:N/B\longrightarrow M$. If $\pi:N\longrightarrow N/B$ is
the canonical map, then the }\textit{\textcolor{black}{R}}\textcolor{black}{-homomorphism
$\beta=\gamma\circ\pi:N\longrightarrow M$ is an extension of $\alpha$
such that if $a\in A$, then $\beta(a)=\gamma\circ\pi(a)=\gamma(a+B)=\tilde{\alpha}(a+B)=\alpha(a)$
as desired. }\end{proof}

\begin{cor}\label{Corollary:(3.2)} \noindent (1) Let $M$ be a semisimple right $R$-module. If $M$ is a strongly ss-injective, then $M$ is small injective.

\noindent (2) If every simple right $R$-module is strongly ss-injective, then $R$ is semiprimitive. \end{cor}
\begin{proof}
(1) By Proposition~\ref{Proposition:(3.1)}.

(2) By (1) and applying \cite[Theorem 2.8]{19ThQu09}.\end{proof}

\begin{rem}\label{Remark:(3.3)} \emph{The converse of Corollary~\ref{Corollary:(3.2)}  is not true (see Example~\ref{Example:(3.8)})}.\end{rem}

\begin{thm}\label{Theorem:(3.4)} If $\mathit{M}$ is a strongly ss-injective (or just ss-$\mathit{E}(M)$-injective) right $R$-module,
then for every semisimple small submodule $\mathit{A}$ of $\mathit{M}$, there is an injective $R$-module $\mathit{E}_{A}$ such that $\mathit{M}=E_{A}\oplus T_{A}$ where $\mathit{T}_{A}$ is a submodule of $M$ with $\mathit{T}_{A}\cap A=0$. Moreover, if $\mathit{A}\neq0$, then $\mathit{E}_{A}$ can be taken $\mathit{A}\leq^{ess}E_{A}$. \end{thm}
\begin{proof}
\textcolor{black}{Let }\textit{\textcolor{black}{A}}\textcolor{black}{{}
be a semisimple small submodule of }\textit{\textcolor{black}{M}}\textcolor{black}{.
If $\mathit{A}=0$, we are done by taking $\mathit{E}_{A}=0$ and
$\mathit{T}_{A}=M$. Suppose that $\mathit{A}\neq0$ and let}
\textcolor{black}{  $\mathit{i}_{1}$, $\mathit{i}_{2}$ and $\mathit{i}_{3}$
be inclusion maps and $D_{A}=E(A)$ be the injective hull of }\textit{\textcolor{black}{A}}\textcolor{black}{{}
in $\mathit{E}(M)$. Since }\textit{\textcolor{black}{M}}\textcolor{black}{{}
is strongly ss-injective, thus }\textit{\textcolor{black}{M}}\textcolor{black}{{}
is ss-$\mathit{E}(M)$-injective. Since }\textit{\textcolor{black}{A}}\textcolor{black}{{}
is a semisimple small submodule of }\textit{\textcolor{black}{M}}\textcolor{black}{,
it follows from \cite[Lemma 5.1.3(a)]{9Kas82} that }\textit{\textcolor{black}{A}}\textcolor{black}{{}
is a semisimple small submodule in $\mathit{E}(M)$ and hence there
exists an }\textit{\textcolor{black}{R}}\textcolor{black}{-homomorphism
$\mathit{\alpha}:E(M)\longrightarrow M$ such that $\alpha i_{2}i_{1}=i_{3}$.
Put $\beta=\alpha i_{2}$, thus $\beta:D_{A}\longrightarrow M$ is
an extension of $i_{3}$. Since $\mathit{A}\leq^{ess}D_{A}$, thus
$\beta$ is a monomorphism. Put $\mathit{E}_{A}=\beta(D_{A})$. Since
$\mathit{E}_{A}$ is an injective submodule of }\textit{\textcolor{black}{M}}\textcolor{black}{,
thus $\mathit{M}=E_{A}\oplus T_{A}$ for some submodule $T_{A}$ of
}\textit{\textcolor{black}{M}}\textcolor{black}{. Since $\beta(A)=A$,
thus $A\subseteq\beta(D_{A})=E_{A}$ and this means that $T_{A}\cap A=0$.
Moreover, define $\tilde{\beta}=\beta:D_{A}\longrightarrow E_{A}$,
thus $\tilde{\beta}$ is an isomorphism. Since $A\leq^{ess}D_{A}$,
thus $\tilde{\beta}(A)\leq^{ess}E_{A}$. But $\tilde{\beta}(A)=\beta(A)=A$,
so $A\leq^{ess}E_{A}$.}\end{proof}

\begin{cor}\label{Corollary:(3.5)}  If $M$ is a right $R$-module has a semisimple small submodule $A$ such that $A\leq^{ess}M$, then the following conditions are equivalent.

\noindent (1) $M$ is injective.

\noindent (2) $M$ is strongly ss-injective.

\noindent (3) $M$ is ss-$E(M)$-injective. \end{cor}
\begin{proof}
\textcolor{black}{(1)$\Rightarrow$(2) and (2)$\Rightarrow$(3) are
obvious.}

\textcolor{black}{(3)$\Rightarrow$(1) By Theorem~\ref{Theorem:(3.4)}, we can write
$M=E_{A}\oplus T_{A}$ where $E_{A}$ injective and $T_{A}\cap A=0$.
Since $A\leq^{ess}M$, thus $T_{A}=0$ and hence $M=E_{A}$. Therefore,
}\textit{\textcolor{black}{M}}\textcolor{black}{{} is an injective }\textit{\textcolor{black}{R}}\textcolor{black}{-module.}\end{proof}

\begin{example}\label{Example:(3.6)}
\emph{\textcolor{black}{$\mathbb{Z}_{4}$ as $\mathbb{Z}$-module is not
strongly ss-injective. In particular, $\mathbb{Z}_{4}$ is not ss-$\mathbb{Z}_{2^{\infty}}$-injective.}}\end{example}
\begin{proof}
\textcolor{black}{Assume that $\mathbb{Z}_{4}$ is strongly ss-injective
$\mathbb{Z}$-module. Let $A=<2>=\left\{ 0,2\right\} $. It is clear
that }\textit{\textcolor{black}{A}}\textcolor{black}{{} is a semisimple
small and essential submodule of $\mathbb{Z}_{4}$ as $\mathbb{Z}$-module.
Thus by Corollary~\ref{Corollary:(3.5)}   we have that $\mathbb{Z}_{4}$ is injective
$\mathbb{Z}$-module and this is a contradiction. Thus $\mathbb{Z}_{4}$
as $\mathbb{Z}$-module is not strongly ss-injective. Since $E(\mathbb{Z}_{2^{2}})=\mathbb{Z}_{2^{\infty}}$
as $\mathbb{Z}$-module, thus $\mathbb{Z}_{4}$ is not ss-$\mathbb{Z}_{2^{\infty}}$-injective,
by Corollary~\ref{Corollary:(3.5)}.}\end{proof}

\begin{cor}\label{Corollary:(3.7)} Let $M$ be a right $R$-module such that soc$(M)\cap J(M)$ is small submodule in $M$ (in particular, if $M$
is finitely generated). If $M$ is strongly ss-injective, then $M=E\oplus T$, where $E$ is injective and $T\cap$ soc$(M)\cap J(M)=0$. Moreover, if soc$(M)\cap J(M)\neq0$, then we can take soc$(M)\cap J(M)\leq^{ess}E$. \end{cor}
\begin{proof}
\textcolor{black}{By taking $A=$ soc$(M)\cap J(M)$ and applying Theorem~\ref{Theorem:(3.4)}}
\end{proof}

The following example shows that the converse of Theorem~\ref{Theorem:(3.4)} and Corollary~\ref{Corollary:(3.7)}   is not true.

\begin{example}\label{Example:(3.8)}
\emph{\textcolor{black}{Let $M=\mathbb{Z}_{6}$ as $\mathbb{Z}$-module.
Since $J(M)=0$ and soc$(M)=M$, thus soc$(M)\cap J(M)=0$. So, we
can write $M=0\oplus M$ with $M\cap($soc$(M)\cap J(M))=0$. Let $N=\mathbb{Z}_{8}$
as $\mathbb{Z}$-module. Since $J(N)=<\bar{2}>$ and soc$(N)=<\bar{4}>$.
Define $\gamma:$ soc$(N)\cap J(N)\longrightarrow M$ by \,\, $\gamma(\bar{4})=\bar{3}$,
\,\, thus $\gamma$ \,\,is \,\,a \,\,\,\,$\mathbb{Z}$-homomorphism. Assume that }\textit{\textcolor{black}{M}}\textcolor{black}{{}
is strongly ss-injective, thus }\textit{\textcolor{black}{M}}\textcolor{black}{{}
is ss-}\textit{\textcolor{black}{N}}\textcolor{black}{-injective,
so there exists $\mathbb{Z}$-homomorphism $\beta:N\longrightarrow M$
such that $\beta\circ i=\gamma$, where }\textit{\textcolor{black}{i}}\textcolor{black}{{}
is the inclusion map from soc$(N)\cap J(N)$ to }\textit{\textcolor{black}{N}}\textcolor{black}{.
Since $\beta(J(N))\subseteq J(M)$, thus $\bar{3}=\gamma(\bar{4})=\beta(\bar{4})\in\beta(J(N))\subseteq J(M)=0$
and this contradiction, so }\textit{\textcolor{black}{M}}\textcolor{black}{{}
is not strongly ss-injective $\mathbb{Z}$-module.}}\end{example}

\begin{cor}\label{Corollary:(3.9)} The following statements are equivalent:

\noindent (1) \emph{soc}$(M)\cap J(M)=0$, for all right $R$-module $M$.

\noindent (2) Every right $R$-module is strongly ss-injective.

\noindent (3) Every  simple right $R$-module is strongly ss-injective.
\end{cor}
\begin{proof}
By Proposition~\ref{Proposition:(2.11)}.
\end{proof}

\subparagraph*{\textmd{Recall that a ring }\textmd{\textit{R}}\textmd{ is called
a right }\textmd{\textit{V}}\textmd{-ring (}\textmd{\textit{GV}}\textmd{-ring}\textmd{\textit{,
SI}}\textmd{-ring, respectively) if every simple (simple singular,
singular, respectively) right }\textmd{\textit{R}}\textmd{-module
is injective. A right }\textmd{\textit{R}}\textmd{-module }\textmd{\textit{M}}\textmd{
is called strongly s-injective if every }\textmd{\textit{R}}\textmd{-homomorphism
from }\textmd{\textit{K}}\textmd{ to }\textmd{\textit{M}}\textmd{
extends to }\textmd{\textit{N}}\textmd{ for every right }\textmd{\textit{R}}\textmd{-module
}\textmd{\textit{N}}\textmd{, where ${\color{black}K\subseteq Z(N)}$
(see \cite{22Zey14}). }\textmd{\textcolor{black}{A submodule }}\textmd{\textit{\textcolor{black}{K}}}\textmd{\textcolor{black}{{}
of a right }}\textmd{\textit{\textcolor{black}{R}}}\textmd{\textcolor{black}{-module
}}\textmd{\textit{\textcolor{black}{M}}}\textmd{\textcolor{black}{{}
is called }}\textmd{\textit{\textcolor{black}{t}}}\textmd{\textcolor{black}{-essential
in }}\textmd{\textit{\textcolor{black}{M}}}\textmd{\textcolor{black}{{}
(written $K\subseteq^{tes}M$ ) if for every submodule }}\textmd{\textit{\textcolor{black}{L}}}\textmd{\textcolor{black}{{}
of }}\textmd{\textit{\textcolor{black}{M}}}\textmd{\textcolor{black}{,
$K\cap L\subseteq Z_{2}(M)$ implies that $L\subseteq Z_{2}(M)$,}}\textmd{\textit{\textcolor{black}{{}
}}}\textmd{\textcolor{black}{}}\textmd{\textit{\textcolor{black}{{}
M}}}\textmd{\textcolor{black}{{} is said to be }}\textmd{\textit{\textcolor{black}{t}}}\textmd{\textcolor{black}{-semisimple
if for every submodule }}\textmd{\textit{\textcolor{black}{A}}}\textmd{\textcolor{black}{{}
of }}\textmd{\textit{\textcolor{black}{M}}}\textmd{\textcolor{black}{{}
there exists a direct summand }}\textmd{\textit{\textcolor{black}{B}}}\textmd{\textcolor{black}{{}
of }}\textmd{\textit{\textcolor{black}{M}}}\textmd{\textcolor{black}{{}
such that $B\subseteq^{tes}A$ (see \cite{4AsHaTo13}) }}\textmd{. In the next
results, we will give some relations between ss-injectivity and other
injectivities and we provide many new equivalences of }\textmd{\textit{V}}\textmd{-rings,
}\textmd{\textit{GV}}\textmd{-rings, }\textmd{\textit{SI}}\textmd{
rings and }\textmd{\textit{QF}}\textmd{ rings.}}

\begin{lem}\label{Lemma:(3.10)} Let $M/N$ be a semisimple right $R$-module and $C$ any right $R$-module. Then every homomorphism from a right submodule (resp. a right semisimple submodule) $A$ of $M$ to $C$ can be extended to a homomorphism from $M$ to $C$ if and only if every homomorphism from a right submodule (resp. a right semisimple submodule) $B$ of $N$ to $C$ can be extended to a homomorphism from $M$ to $C$.\end{lem}
\begin{proof}
\textcolor{black}{($\Rightarrow$) is obtained directly.}

\textcolor{black}{($\Leftarrow$) Let  }\textit{\textcolor{black}{f}}\textcolor{black}{{}
 be a right }\textit{\textcolor{black}{R}}\textcolor{black}{-homomorphism
from a right submodule }\textit{\textcolor{black}{A}}\textcolor{black}{{}
of }\textit{\textcolor{black}{M}}\textcolor{black}{{} to }\textit{\textcolor{black}{C}}\textcolor{black}{.
Since $M/N$ is semisimple, thus there exists a right submodule }\textit{\textcolor{black}{L}}\textcolor{black}{{}
of }\textit{\textcolor{black}{M}}\textcolor{black}{{} such that $A+L=M$
and $A\cap L\leq N$ (see \cite[Proposition 2.1]{11Lom99}). Thus there
exists a right }\textit{\textcolor{black}{R}}\textcolor{black}{-homomorphism
$g:M\longrightarrow C$ such that $g(x)=f(x)$ for all $x\in A\cap L$.
Define $h:M\longrightarrow C$ such that for any $x=a+\ell$, $a\in A$,
$\ell\in L$, $h(x)=f(a)+g(\ell)$. Thus }\textit{\textcolor{black}{h}}\textcolor{black}{{}
is a well define }\textit{\textcolor{black}{R}}\textcolor{black}{-homomorphism,
because if $a_{1}+\ell_{1}=a_{2}+\ell_{2}$, $a_{i}\in A$, $\ell_{i}\in L$,
$i=1,2$, then $a_{1}-a_{2}=\ell_{2}-\ell_{1}\in A\cap L$, that is
$f(a_{1}-a_{2})=g(\ell_{2}-\ell_{1})$ which leads to $h(a_{1}+\ell_{1})=h(a_{2}+\ell_{2})$.
Thus }\textit{\textcolor{black}{h }}\textcolor{black}{is a well define
}\textit{\textcolor{black}{R}}\textcolor{black}{-homomorphism and
extension of }\textit{\textcolor{black}{f}}\textcolor{black}{.}\end{proof}

\begin{cor}\label{Corollary:(3.11)} For right $R$-modules $M$ and $N$, then the following hold:

\noindent (1) If $M$  is finitely generated and $M/J(M)$ is semisimple right $R$-module, then $N$  is right \emph{soc}-$M$-injective
if and only if $N$ is right ss-$M$-injective.

\noindent (2) If $M/$\emph{soc}$(M)$ is a semisimple right $R$-module, then $N$  is \emph{soc}-$M$-injective if and only if $N$ is $M$-injective.

\noindent (3) If $R/S_{r}$ is semisimple right $R$-module, then $N$ is \emph{soc}-injective if and only if $N$ is injective.

\noindent (4) If $R/S_{r}$ is semisimple right $R$-module, then $N$ is ss-injective if and only if $N$  is small injective. \end{cor}
\begin{proof}
\textcolor{black}{(1). ($\Rightarrow$) Clear.}

\textcolor{black}{($\Leftarrow$) Since }\textit{\textcolor{black}{N}}\textcolor{black}{{}
is a right ss-}\textit{\textcolor{black}{M}}\textcolor{black}{-injective,
thus every homomorphism from a semisimple small submodule of }\textit{\textcolor{black}{M}}\textcolor{black}{{}
to }\textit{\textcolor{black}{N}}\textcolor{black}{{} extends to }\textit{\textcolor{black}{M}}\textcolor{black}{.
Since }\textit{\textcolor{black}{M}}\textcolor{black}{{} is finitely
generated, thus $J(M)\ll M$ and hence every homomorphism from any
semisimple submodule of }\textit{\textcolor{black}{$J(M)$}}\textcolor{black}{{}
to }\textit{\textcolor{black}{N}}\textcolor{black}{{} extends to }\textit{\textcolor{black}{M}}\textcolor{black}{.
Since $M/J(M)$ is semisimple. Thus every homomorphism from any semisimple
submodule of }\textit{\textcolor{black}{M}}\textcolor{black}{{} to }\textit{\textcolor{black}{N}}\textcolor{black}{{}
extends to }\textit{\textcolor{black}{M}}\textcolor{black}{{} by Lemma~\ref{Lemma:(3.10)}. Therefore }\textit{\textcolor{black}{N}}\textcolor{black}{{} is
a soc-}\textit{\textcolor{black}{M}}\textcolor{black}{-injective right
}\textit{\textcolor{black}{R}}\textcolor{black}{-module.}

\textcolor{black}{(2). ($\Rightarrow$) Since }\textit{\textcolor{black}{N}}\textcolor{black}{{}
is soc-}\textit{\textcolor{black}{M}}\textcolor{black}{-injective.
Thus every homomorphism from any submodule of soc$(M)$ to }\textit{\textcolor{black}{N}}\textcolor{black}{{}
extends to }\textit{\textcolor{black}{M}}\textcolor{black}{. Since
$M/$soc$(M)$ is semisimple, thus Lemma~\ref{Lemma:(3.10)} implies that every homomorphism
from any submodule of }\textit{\textcolor{black}{M}}\textcolor{black}{{}
to }\textit{\textcolor{black}{N}}\textcolor{black}{{} extends to }\textit{\textcolor{black}{M}}\textcolor{black}{.
Hence }\textit{\textcolor{black}{N}}\textcolor{black}{{} is }\textit{\textcolor{black}{M}}\textcolor{black}{-injective.}

\textcolor{black}{($\Leftarrow$) Clear.}

\textcolor{black}{(3) By (2).}

\textcolor{black}{(4) Since $R/S_{r}$ is semisimple right }\textit{\textcolor{black}{R}}\textcolor{black}{-module,
thus $J(R/S_{r})=0$. By \cite[Theorem 9.1.4(b)]{9Kas82}, we have $J\subseteq S_{r}$
and hence $J=J\cap S_{r}$. Thus }\textit{\textcolor{black}{N}}\textcolor{black}{{}
is ss-injective if and only if }\textit{\textcolor{black}{N}}\textcolor{black}{{}
is small injective.}\end{proof}

\begin{cor}\label{Corollary:(3.12)}  Let $R$ be a semilocal ring, then $S_{r}\cap J$ is finitely generated if and only if $S_{r}$ is finitely generated. \end{cor}
\begin{proof}
\textcolor{black}{Suppose that $S_{r}\cap J$ is finitely generated.
By Corollary~\ref{Corollary:(2.27)}, every direct sum of soc-injective right }\textit{\textcolor{black}{R}}\textcolor{black}{-modules
is ss-injective. Thus it follows from Corollary~\ref{Corollary:(3.11)}(1)  and \cite[Corollary 2.11]{2AmYoZe05}  that $S_{r}$ is finitely generated. The converse
is clear.}\end{proof}

\begin{thm}\label{Theorem:(3.13)} If $R$ is a right perfect ring, then a right $R$-module $M$ is strongly soc-injective if and only if $M$ is strongly ss-injective.\end{thm}
\begin{proof}
\textcolor{black}{($\Rightarrow$) Clear.}

\textcolor{black}{($\Leftarrow$) Let }\textit{\textcolor{black}{R}}\textcolor{black}{{}
be a right perfect ring and }\textit{\textcolor{black}{M}}\textcolor{black}{{}
be a strongly ss-injective right $R$-module. By \cite[Theorem 3.8]{11Lom99}, }\textit{\textcolor{black}{R}}\textcolor{black}{{}
is a semilocal ring and hence by \cite[Theorem 3.5]{11Lom99}, we have every
right }\textit{\textcolor{black}{R}}\textcolor{black}{-module }\textit{\textcolor{black}{N}}\textcolor{black}{{}
is semilocal and hence $N/J(N)$ is semisimple right }\textit{\textcolor{black}{R}}\textcolor{black}{-module.
Since }\textit{\textcolor{black}{R}}\textcolor{black}{{} is a right
perfect ring, thus the Jacobson radical of every right }\textit{\textcolor{black}{R}}\textcolor{black}{-module
is small by \cite[Theorem 4.3 and 4.4, p. 69]{7ChDiMa05}. Thus $N/J(N)$ is
semisimple and $J(N)\ll N$, for any $N\in$ Mod-}\textit{\textcolor{black}{R}}\textcolor{black}{.
Since }\textit{\textcolor{black}{M}}\textcolor{black}{{} is strongly
ss-injective, thus every homomorphism from a semisimple small submodule
of }\textit{\textcolor{black}{N}}\textcolor{black}{{} to }\textit{\textcolor{black}{M}}\textcolor{black}{{}
extends to }\textit{\textcolor{black}{N}}\textcolor{black}{, for every
$N\in$ Mod-}\textit{\textcolor{black}{R}}\textcolor{black}{, and this
implies that every homomorphism from any semisimple submodule of $J(N)$
to }\textit{\textcolor{black}{M}}\textcolor{black}{{} extends to }\textit{\textcolor{black}{N}}\textcolor{black}{,
for every $N\in$ Mod-}\textit{\textcolor{black}{R}}\textcolor{black}{.
Since $N/J(N)$ is semisimple right }\textit{\textcolor{black}{R}}\textcolor{black}{-module,
for every $N\in$ Mod-}\textit{\textcolor{black}{R}}\textcolor{black}{.
Thus Lemma~\ref{Lemma:(3.10)}  implies that every homomorphism from any semisimple
submodule of }\textit{\textcolor{black}{N}}\textcolor{black}{{} to }\textit{\textcolor{black}{M}}\textcolor{black}{{}
extends to }\textit{\textcolor{black}{N}}\textcolor{black}{, for every
$N\in$ Mod-}\textit{\textcolor{black}{R}}\textcolor{black}{{} and hence
}\textit{\textcolor{black}{M}}\textcolor{black}{{} is strongly soc-injective.}\end{proof}

\begin{cor}\label{Corollary:(3.14)} A ring $R$ is  $QF$  ring if and only if every strongly ss-injective right $R$-module is projective. \end{cor}
\begin{proof}
\textcolor{black}{($\Rightarrow$) If }\textit{\textcolor{black}{R}}\textcolor{black}{{}
is }\textit{\textcolor{black}{QF}}\textcolor{black}{{} ring, then }\textit{\textcolor{black}{R}}\textcolor{black}{{}
is a right perfect ring, so by Theorem~\ref{Theorem:(3.13)}  and \cite[Proposition 3.7]{2AmYoZe05} we have every strongly ss-injective right $R$-module is projective.}

\textcolor{black}{($\Leftarrow$) By hypothesis we have every injective
right }\textit{\textcolor{black}{R}}\textcolor{black}{-module is projective
and hence }\textit{\textcolor{black}{R}}\textcolor{black}{{} is }\textit{\textcolor{black}{QF}}\textcolor{black}{{}
ring (see for instance \cite[Proposition 12.5.13]{6Bla11}).}\end{proof}

\begin{thm}\label{Theorem:(3.15)} The following statements are equivalent for a ring $R$.

\noindent (1) Every direct sum of strongly ss-injective right $R$-modules is injective.

\noindent (2) Every direct sum of strongly soc-injective right $R$-modules is injective.

\noindent (3) $R$ is right artinian. \end{thm}

\begin{proof}
\textcolor{black}{(1)$\Rightarrow$(2) Clear.}

\textcolor{black}{(2)$\Rightarrow$(3) Since every direct sum of strongly
soc-injective right }\textit{\textcolor{black}{R}}\textcolor{black}{-modules
is injective, thus }\textit{\textcolor{black}{R}}\textcolor{black}{{}
is right noetherian and right semiartinian by \cite[Theorem 3.3 and Theorem 3.6]{2AmYoZe05}, so it follows from \cite[Proposition 5.2, p.189]{18Ste75} that }\textit{\textcolor{black}{R}}\textcolor{black}{{} is right artinian.}

\textcolor{black}{(3)$\Rightarrow$(1) By hypothesis, }\textit{\textcolor{black}{R}}\textcolor{black}{{}
is right perfect and right noetherian. It follows from Theorem~\ref{Theorem:(3.13)}
and \cite[Theorem 3.3]{2AmYoZe05} that every direct sum of strongly ss-injective
right }\textit{\textcolor{black}{R}}\textcolor{black}{-modules is
strongly soc-injective. Since }\textit{\textcolor{black}{R}}\textcolor{black}{{}
is right semiartinian, so \cite[Theorem 3.6]{2AmYoZe05}  implies that every
direct sum of strongly ss-injective right }\textit{\textcolor{black}{R}}\textcolor{black}{-modules
is injective .}\end{proof}

\begin{thm}\label{Theorem:(3.16)} If $R$ is a right $t$-semisimple, then a right $R$-module $M$ is injective if and only if $M$ is strongly s-injective. \end{thm}
\begin{proof}
\textcolor{black}{($\Rightarrow$) Obvious.}

\textcolor{black}{($\Leftarrow$) Since }\textit{\textcolor{black}{M}}\textcolor{black}{{}
is strongly s-injective, thus $Z_{2}(M)$ is injective by \cite[Proposition 3, p.27]{22Zey14}. Thus every homomorphism $f:K\longrightarrow M$, where
$K\subseteq Z_{2}^{r}$ extends to }\textit{\textcolor{black}{R}}\textcolor{black}{{}
by \cite[Lemma 1, p.26]{22Zey14}. Since }\textit{\textcolor{black}{R}}\textcolor{black}{{}
is a right }\textit{\textcolor{black}{t}}\textcolor{black}{-semisimple,
thus $R/Z_{2}^{r}$ is a right semisimple (see \cite[Theorem 2.3]{4AsHaTo13}).
So by applying Lemma~\ref{Lemma:(3.10)}, we conclude that }\textit{\textcolor{black}{M}}\textcolor{black}{{}
is injective.}\end{proof}

\begin{cor}\label{Corollary:(3.17)} The following statements  are equivalent for a ring  $R$.

\noindent (1) $R$ is right  $SI$  and right $t$-semisimple.

\noindent (2) $R$ is semisimple. \end{cor}
\begin{proof}
\textcolor{black}{(1)$\Rightarrow$(2). Since }\textit{\textcolor{black}{R}}\textcolor{black}{{}
is a right }\textit{\textcolor{black}{SI}}\textcolor{black}{{} ring,
thus every right }\textit{\textcolor{black}{R}}\textcolor{black}{-module
is strongly s-injective by \cite[Theorem 1, p.29]{22Zey14}. By Theorem~\ref{Theorem:(3.16)},
we have every right }\textit{\textcolor{black}{R}}\textcolor{black}{-module
is injective and hence }\textit{\textcolor{black}{R}}\textcolor{black}{{}
is semisimple ring.}

\textcolor{black}{(2)$\Rightarrow$(1). Clear.}\end{proof}

\begin{cor}\label{Corollary:(3.18)} If $R$ is a right $t$-semisimple ring, then $R$ is right $V$-ring if and only if $R$ is right $GV$-ring. \end{cor}
\begin{proof}
\textcolor{black}{($\Rightarrow$). Clear.}

\textcolor{black}{($\Leftarrow$). By \cite[Proposition 5, p.28]{22Zey14}
and Theorem~\ref{Theorem:(3.16)}.}\end{proof}

\begin{cor}\label{Corollary:(3.19)} If $R$ is a right $t$-semisimple ring, then  $R/S_{r}$  is noetherian right $R$-module if and only if $R$ is right noetherian. \end{cor}
\begin{proof}
If \textcolor{black}{$R/S_{r}$ is a noetherian right }\textit{\textcolor{black}{R}}\textcolor{black}{-module,
thus every direct sum of injective right }\textit{\textcolor{black}{R}}\textcolor{black}{-modules
is strongly s-injective by \cite[Proposition 6]{22Zey14}. Since }\textit{\textcolor{black}{R}}\textcolor{black}{{}
is right }\textit{\textcolor{black}{t}}\textcolor{black}{-semisimple,
so it follows from Theorem~\ref{Theorem:(3.16)} that every direct sum of injective
right }\textit{\textcolor{black}{R}}\textcolor{black}{-modules is
injective and hence }\textit{\textcolor{black}{R}}\textcolor{black}{{}
is right noetherian. The converse is clear.}\end{proof}

\section{SS-Injective Rings}

\subparagraph*{\textmd{\textcolor{black}{We recall that the dual of a right }}\textmd{\textit{\textcolor{black}{R}}}\textmd{\textcolor{black}{-module
}}\textmd{\textit{\textcolor{black}{M}}}\textmd{\textcolor{black}{{}
is $M^{d}=$Hom$_{R}(M,R_{R})$ and clearly that $M^{d}$ is a left }}\textmd{\textit{\textcolor{black}{R}}}\textmd{\textcolor{black}{-module.}}}

\begin{prop}\label{Proposition:(4.1)} The following statements are equivalent for a ring $R$.

\noindent (1) $R$ is a right ss-injective ring. 

\noindent (2) If $K$ is a semisimple right $R$-module, $P$ and $Q$ are finitely generated projective right $R$-modules,
$\beta:K\longrightarrow P$ is an $R$-monomorphism with $\beta(K)\ll P$ and $f:K\longrightarrow Q$ is an $R$-homomorphism,
then $f$  can be extended to an $R$-homomorphism $h:P\longrightarrow Q$.

\noindent (3) If $M$  is a right semisimple $R$-module and $f$ is a nonzero monomorphism from $M$ to $R_{R}$ with $f(M)\ll R_{R}$, then $M^{d}=Rf$.

\end{prop}
\begin{proof}
\textcolor{black}{(2)$\Rightarrow$(1) Clear.}

\textcolor{black}{(1)$\Rightarrow$(2)} Since \textit{Q} is finitely generated, there is an \textit{R}-epimorphism
\textit{\textcolor{black}{$\alpha_{1}:R^{n}\longrightarrow Q$ }}for
some \textit{\textcolor{black}{$n\in\mathbb{Z}^{+}$.}} Since \textit{Q}
is projective, there is an \textit{R}-homomorphism \textit{\textcolor{black}{$\alpha_{2}:Q\longrightarrow R^{n}$}}
such that \textit{\textcolor{black}{$\alpha_{1}\alpha_{2}=I_{Q}$.}}
 Define \textit{\textcolor{black}{$\tilde{\beta}:K\longrightarrow\beta(K)$}} by \textit{\textcolor{black}{$\tilde{\beta}(a)=\beta(a)$
}}for all ${\color{black}a\in K}$. Since \textit{R} is a right ss-injective
ring (by hypothesis), it follows from Proposition~\ref{Proposition:(2.8)}  and Corollary~\ref{Corollary:(2.4)}(1)
 that \textit{\textcolor{black}{$R^{n}$ }}is a right ss-\textit{P}-injective
\textit{R}-module. So there exists an \textit{R}-homomorphism \textit{\textcolor{black}{$h:P\longrightarrow R^{n}$
}}such that \textit{\textcolor{black}{$hi=\alpha_{2}f$ $\tilde{\beta}^{-1}$}}. Put
${\color{black}g=\alpha_{1}h:P\longrightarrow Q}$. Thus \textit{\textcolor{black}{$gi=(\alpha_{1}h)i=\alpha_{1}(\alpha_{2}f$
$\tilde{\beta}^{-1})=f$ $\tilde{\beta}^{-1}$}} and hence \textit{\textcolor{black}{$(g\beta)(a)=g(i(\beta(a)))=(f$
$\tilde{\beta}^{-1})(\beta(a))=f$ $(a)$}} for all ${\color{black}a\in K}$.
Therefore, there is an \textit{R}-homomorphism \textit{\textcolor{black}{$g:P\longrightarrow Q$}}
such that \textit{\textcolor{black}{$g\beta=f$. }}

\textcolor{black}{(1)$\Rightarrow$(3) Let $g\in M^{d}$, we have
$gf^{-1}:f(M)\rightarrow R_{R}$. Since $f(M)$ is a semisimple
small right ideal of }\textit{\textcolor{black}{R}}\textcolor{black}{{}
and }\textit{\textcolor{black}{R}}\textcolor{black}{{} is a right ss-injective
ring (by hypothesis), thus $gf^{-1}=a.$ for some $a\in R$.
Therefore, $g=af$ and hence $M^{d}=Rf$.}

\textcolor{black}{(3)$\Rightarrow$(1) Let $f:K\longrightarrow R$
be a right }\textit{\textcolor{black}{R}}\textcolor{black}{-homomorphism,
where }\textit{\textcolor{black}{K}}\textcolor{black}{{} is a semisimple
small right ideal of }\textit{\textcolor{black}{R}}\textcolor{black}{{}
and let $i:K\longrightarrow R$ be the inclusion map, thus by (2)
we have $K^{d}=Ri$ and hence $f=ci$ in $K^{d}$ for some $c\in R$.
Thus there is $c\in R$ such that $f(a)=ca$ for all $a\in K$ and
this implies that }\textit{\textcolor{black}{R}}\textcolor{black}{{}
is a right ss-injective ring.}

\end{proof}

\begin{example}\label{Example:(4.2)}
\noindent\emph{(1) Every  universally mininjective ring is ss-injective, but
not conversely (see Example~\ref{Example:(5.7)}).}

\noindent\emph{(2) The two classes of universally mininjective rings and soc-injective
rings are different (see Example~\ref{Example:(5.7)} and Example~\ref{Example:(5.8)}).}\end{example}

\begin{cor}\label{Corollary:(4.3)} Let $R$ be a right ss-injective ring. Then:

\noindent (1) $R$ is a right mininjective ring.

\noindent (2) $lr(a)=Ra$, for all $a\in S_{r}\cap J$.

\noindent (3) $r(a)\subseteq r(b)$, $a\in S_{r}\cap J$, $b\in R$ implies $Rb\subseteq Ra$.

\noindent (4) $l(bR\cap r(a))=l(b)+Ra$, for all $a\in S_{r}\cap J$, $b\in R$.

\noindent (5) $l(K_{1}\cap K_{2})=l(K_{1})+l(K_{2})$, for all semisimple small right ideals $K_{1}$ and $K_{2}$ of $R$.\end{cor}
\begin{proof}
(1) By Lemma~\ref{Lemma:(2.5)}.

(2), (3),(4) and (5) are obtained by Lemma~\ref{Lemma:(2.13)}.
\end{proof}

\subparagraph*{\textmd{The following is an example of a right mininjective ring
which is not right ss-injective.}}

\begin{example}\label{Example:(4.4)}
\emph{(The  Bj\"{o}rk  Example \cite[Example 2.5]{15NiYu03})\textcolor{black}{.
Let }\textit{\textcolor{black}{F}}\textcolor{black}{{} be a field and
let $a\mapsto\bar{a}$ be an isomorphism $F\longrightarrow\bar{F}\subseteq F$,
where the subfield $\bar{F}\neq F$. Let }\textit{\textcolor{black}{R}}\textcolor{black}{{}
denote the left vector space on basis $\{1,  t\}$, and make }\textit{\textcolor{black}{R}}\textcolor{black}{{} into an
}\textit{\textcolor{black}{F}}\textcolor{black}{-algebra by defining
$t^{2}=0$ and $ta=\bar{a}t$ for all $a\in F$. By \cite[Example 2.5]{15NiYu03} we have }\textit{\textcolor{black}{R}}\textcolor{black}{{} is
a right mininjective local ring. It is mentioned in \cite[Example 4.15]{2AmYoZe05}, that }\textit{\textcolor{black}{R}}\textcolor{black}{{}
is not right soc-injective. Since $R$ is a local ring, thus by Corollary~\ref{Corollary:(3.11)}(1), }\textit{\textcolor{black}{R}}\textcolor{black}{{} is not
right ss-injective ring.}}\end{example}

\begin{thm}\label{Theorem:(4.5)} Let $R$ be a right ss-injective ring. Then: 

\noindent (1) $S_{r}\cap J\subseteq Z_{r}$. 

\noindent (2)  If the ascending chain $r(a_{1})\subseteq r(a_{2}a_{1})\subseteq...$
terminates for any sequence $a_{1},a_{2},...$ in $Z_{r}\cap S_{r}$,
then $S_{r}\cap J$ is right t-nilpotent and $S_{r}\cap J=Z_{r}\cap S_{r}$. \end{thm}
\begin{proof}
\textcolor{black}{(1) Let $a\in S_{r}\cap J$ and $bR\cap r(a)=0$
for any $b\in R$. By Corollary~\ref{Corollary:(4.3)}(4), $l(b)+Ra=l(bR\cap r(a))=l(0)=R$,
so $l(b)=R$ because $a\in J$, implies that $b=0$. Thus $r(a)\subseteq^{ess}R_{R}$
and hence $S_{r}\cap J\subseteq Z_{r}$.}

\textcolor{black}{(2) For any sequence $x_{1},x_{2},...$ in $Z_{r}\cap S_{r}$,
we have $r(x_{1})\subseteq r(x_{2}x_{1})\subseteq...$ . By hypothesis,
there exists $m\in\mathbb{N}$ such that $r(x_{m}...x_{2}x_{1})=r(x_{m+1}x_{m}...x_{2}x_{1})$.
If $x_{m}...x_{2}x_{1}\neq0$, then $(x_{m}...x_{2}x_{1})R\cap r(x_{m+1})\neq0$
and hence $0\neq x_{m}...x_{2}x_{1}r\in r(x_{m+1})$ for some $r\in R$.
Thus $x_{m+1}x_{m}...x_{2}x_{1}r=0$ and this implies that $x_{m}...x_{2}x_{1}r=0$,
a contradiction. Thus $Z_{r}\cap S_{r}$ is right t-nilpotent, so
$Z_{r}\cap S_{r}\subseteq J$ . Therefore, $S_{r}\cap J=Z_{r}\cap S_{r}$
by (1).}\end{proof}

\begin{prop}\label{Proposition:(4.6)} Let $R$ be a right ss-injective ring. Then: 

\noindent (1) If $Ra$ is a simple left ideal of $R$,
then soc$(aR)\cap J(aR)$ is zero or simple. 

\noindent (2) $rl(S_{r}\cap J)=S_{r}\cap J$ if and only if $rl(K)=K$ for all semisimple small right ideals $K$ of $R$.\end{prop}
\begin{proof}
(1)\textcolor{black}{{} Suppose that soc$(aR)\cap J(aR)$ is a nonzero.
Let $x_{1}R$ and $x_{2}R$ be any simple small right ideals of }\textit{\textcolor{black}{R}}\textcolor{black}{{}
with $x_{i}\in aR$, $i=1,2$. If $x_{1}R\cap x_{2}R=0$, then by
Corollary~\ref{Corollary:(4.3)}(5) $l(x_{1})+l(x_{2})=R$. Since $x_{i}\in aR$, thus
$x_{i}=ar_{i}$ for some $r_{i}\in R$, $i=1,2$, that is $l(a)\subseteq l(ar_{i})=l(x_{i})$,
$i=1,2$. Since $Ra$ is a simple, then $l(a)\subseteq^{max}R$, that
is $l(x_{1})=l(x_{2})=l(a)$. Therefore, $l(a)=R$ and hence $a=0$
and this contradicts the minimality of $Ra$. Thus soc$(aR)\cap J(aR)$
is simple.}

\textcolor{black}{(2) Suppose that $rl(S_{r}\cap J)=S_{r}\cap J$
and let }\textit{\textcolor{black}{K}}\textcolor{black}{{} be a semisimple
small right ideal of }\textit{\textcolor{black}{R}}\textcolor{black}{,
trivially we have $K\subseteq rl(K)$. If $K\cap xR=0$ for some $x\in rl(K)$,
then by Corollary~\ref{Corollary:(4.3)}(5) $l(K\cap xR)=l(K)+l(xR)=R$, since $x\in rl(K)\subseteq$$rl(S_{r}\cap J)=S_{r}\cap J$
. If $y\in l(K)$, then $yx=0$, that is $y(xr)=0$ for all $r\in R$
and hence $l(K)\subseteq l(xR)$.Thus $l(xR)=R$, so $x=0$ and this
means that $K\subseteq^{ess}rl(K)$. Since $K\subseteq^{ess}rl(K)\subseteq rl(S_{r}\cap J)=S_{r}\cap J$,
it follows that $K=rl(K)$. The converse is trivial.}\end{proof}

\begin{lem}\label{Lemma:(4.7)} The following statements are equivalent. 

\noindent (1) $rl(K)=K$, for all semisimple small
right ideals $K$  of $R$.

\noindent  (2) $r(l(K)\cap Ra)=K+r(a)$, for all semisimple small right ideals $K$
of $R$ and all $a\in R$. \end{lem}
\begin{proof}
\textcolor{black}{(1)$\Rightarrow$(2). Clearly, $K+r(a)\subseteq r(l(K)\cap Ra)$
by \cite[Proposition 2.16]{3AnFu74}. Now, let $x\in r(l(K)\cap Ra)$ and
$y\in l(aK)$. Then $yaK=0$ and $y\in l(ax)$. Thus $l(aK)\subseteq l(ax)$,
and so $ax\in rl(ax)\subseteq rl(aK)=aK$, since $aK$ is a semisimple
small right ideal of }\textit{\textcolor{black}{R}}\textcolor{black}{.
Hence $ax=ak$ for some $k\in K$, and so $(x-k)\in r(a)$. This leads
to $x\in K+r(a)$, that is $r(l(K)\cap Ra)=K+r(a)$.}

\textcolor{black}{(2)$\Rightarrow$(1). By taking $a=1$.}
\end{proof}

\subparagraph*{\textmd{\textcolor{black}{Recall that a right ideal }}\textmd{\textit{\textcolor{black}{I
}}}\textmd{\textcolor{black}{of }}\textmd{\textit{\textcolor{black}{R}}}\textmd{\textcolor{black}{{}
is said to be lie over a summand of $R_{R}$, if there exists a direct
decomposition $R_{R}=A_{R}\oplus B_{R}$ with $A\subseteq I$ and
$B\cap I\ll R_{R}$ (see \cite{13Nich76}) which leads to $I=A\oplus(B\cap I)$.}}}

\begin{lem}\label{Lemma:(4.8)} Let $K$ be an $m$-generated semisimple right ideal lies over summand of $R_{R}$. If $R$
is right ss-injective, then every homomorphism from $K$ to $R_{R}$ can be extended to an endomorphism of $R_{R}$. \end{lem}
\begin{proof}
\textcolor{black}{Let $\alpha:K\longrightarrow R$ be a right }\textit{\textcolor{black}{R}}\textcolor{black}{-homomorphism.
By hypothesis, $K=eR\oplus B$, for some $e^{2}=e\in R$, where $B$
is an }\textit{\textcolor{black}{m}}\textcolor{black}{-generated semisimple
small right ideal of }\textit{\textcolor{black}{R}}\textcolor{black}{{}
. Now, we need prove that $K=eR\oplus(1-e)B$. Clearly, $eR+(1-e)B$
is a direct sum. Let $x\in K$, then $x=a+b$ for some $a\in eR$,
$b\in B$, so we can write $x=a+eb+(1-e)b$ and this implies that
$x\in eR\oplus(1-e)B$. Conversely, let $x\in eR\oplus(1-e)B$. Thus
$x=a+(1-e)b$, for some $a\in eR$, $b\in B$. We obtain $x=a+(1-e)b=(a-eb)+b\in eR\oplus B$.
It is obvious that $(1-e)B$ is an }\textit{\textcolor{black}{m}}\textcolor{black}{-generated
semisimple small right ideal. Since }\textit{\textcolor{black}{R}}\textcolor{black}{{}
is a right ss-injective, then there exists $\gamma\in End(R_{R})$
such that $\gamma_{|(1-e)B}=\alpha_{|(1-e)B}$ . Define $\beta:R_{R}\longrightarrow R_{R}$
by $\beta(x)=\alpha(ex)+\gamma((1-e)x)$, for all $x\in R$ which
is a well defined $R$-homomorphism. If $x\in K$, then $x=a+b$ where
$a\in eR$ and $b\in(1-e)B$, so $\beta(x)=\alpha(ex)+\gamma((1-e)x)=\alpha(a)+\gamma(b)=\alpha(a)+\alpha(b)=\alpha(x)$
which yields $\beta$ is an extension of $\alpha$. }\end{proof}

\begin{cor}\label{Corollary:(4.9)} Let $R$ be a semiregular ring (or just every finitely generated semisimple right ideal lies over a summand of $R_{R}$). If $R$ is a right ss-injective ring, then every $R$-homomorphism from a finitely generated semisimple right ideal to $R$ extends to $R$.\end{cor}
\begin{proof}
\textcolor{black}{By \cite[Theorem 2.9]{13Nich76}  and Lemma~\ref{Lemma:(4.8)}.}\end{proof}

\begin{cor}\label{Corollary:(4.10)} Let $S_{r}$ be a finitely generated and lie over a summand of $R_{R}$, then $R$ is a right ss-injective ring if and only if $R$ is right soc-injective.
\end{cor}

\subparagraph*{\textmd{\textcolor{black}{Recall that a ring }}\textmd{\textit{\textcolor{black}{R}}}\textmd{\textcolor{black}{{}
is called right minannihilator if every simple right ideal }}\textmd{\textit{\textcolor{black}{K}}}\textmd{\textcolor{black}{{}
of }}\textmd{\textit{\textcolor{black}{R}}}\textmd{\textcolor{black}{{}
is an annihilator; equivalently, if $rl(K)=K$ (see \cite{14NiYo97}).}}}

\begin{lem}\label{Lemma:(4.11)} A ring $R$ is a right minannihilator if and only if  $rl(K)=K$ for any simple small right ideal $K$ of $R$.
\end{lem}

\begin{lem}\label{Lemma:(4.12)} A ring $R$ is a left minannihilator if and only if  $lr(K)=K$ for any simple small left ideal $K$ of $R$.

\end{lem}

\begin{cor}\label{Corollary:(4.13)} Let $R$  be a right ss-injective ring, then the following hold: 

\noindent (1) If $rl(S_{r}\cap J)=S_{r}\cap J$, then $R$ is right minannihilator. 

\noindent (2) If $S_{\ell}\subseteq S_{r}$, then: 

\noindent i)  $S_{\ell}=S_{r}$.

\noindent ii) $R$ is a left minannihilator ring. \end{cor}
\begin{proof}
(1)\textcolor{black}{{} Let $aR$ be a simple small right ideal of }\textit{\textcolor{black}{R}}\textcolor{black}{,
thus $rl(a)=aR$ by Proposition~\ref{Proposition:(4.6)}(2). Therefore, }\textit{\textcolor{black}{R}}\textcolor{black}{{}
is a right minannihilator ring.}

\textcolor{black}{(2) i) Since }\textit{\textcolor{black}{R}}\textcolor{black}{{}
is a right ss-injective ring, thus it is right mininjective and it follows
from \cite[Proposition 1.14 (4)]{14NiYo97}   that $S_{\ell}=S_{r}$ .}

\textcolor{black}{ii) If $Ra$ is a simple small left ideal of }\textit{\textcolor{black}{R}}\textcolor{black}{,
then $lr(a)=Ra$ by Corollary~\ref{Corollary:(4.3)}(2)  and hence }\textit{\textcolor{black}{R}}\textcolor{black}{{}
is a left minannihilator ring.}\end{proof}

\begin{prop}\label{Proposition:(4.14)} The following statements are equivalent for a right ss-injective ring $R$.

\noindent (1) $S_{\ell}\subseteq S_{r}$. 

\noindent (2) $S_{\ell}=S_{r}$. 

\noindent (3) $R$ is a left mininjective ring. \end{prop}
\begin{proof}
\textcolor{black}{(1)$\Rightarrow$(2) By Corollary~\ref{Corollary:(4.13)}(2) (i).}

\textcolor{black}{(2)$\Rightarrow$(3) By Corollary~\ref{Corollary:(4.13)}(2) and \cite[Corollary 2.34]{15NiYu03}, we need only show that }\textit{\textcolor{black}{R}}\textcolor{black}{{}
is right minannihilator ring. Let $aR$ be a simple small right ideal,
then $Ra$ is a simple small left ideal by \cite[Theorem 1.14]{14NiYo97}.
Let $0\neq x\in rl(aR)$, then $l(a)\subseteq l(x)$. Since $l(a)\leq^{max}R$,
thus $l(a)=l(x)$ and hence $Rx$ is simple left ideal, that is $x\in S_{r}$.
Now , if $Rx=Re$ for some $e=e^{2}\in R$, then $e=rx$ for some
$0\neq r\in R$. Since $(e-1)e=0$, then $(e-1)rx=0$, that is $(e-1)ra=0$
and this implies that $ra\in eR$. Thus $raR\subseteq eR$, but $eR$
is semisimple right ideal, so $raR\subseteq^{\oplus}R$ and hence
$ra=0$. Therefore, $rx=0$, that is $e=0$, a contradiction. Thus
$x\in J$ and hence $x\in S_{r}\cap J$. Therefore, $aR\subseteq rl(aR)\subseteq S_{r}\cap J$.
Now, let $aR\cap yR=0$ for some $y\in rl(aR)$, thus $l(aR)+l(yR)=l(aR\cap yR)=R$.
Since $y\in rl(aR)$, thus $l(aR)\subseteq l(yR)$ and hence $l(yR)=R$,
that is $y=0$. Therefore, $aR\subseteq^{ess}rl(aR)$, so $aR=rl(aR)$
as desired.}

\textcolor{black}{(3)$\Rightarrow$(1) Follows from \cite[Corollary 2.34]{15NiYu03}.}
\end{proof}

\subparagraph*{\textmd{\textcolor{black}{Recall that a ring }}\textmd{\textit{\textcolor{black}{R}}}\textmd{\textcolor{black}{{}
is said to be right minfull if it is semiperfect, right mininjective
and soc$(eR)\neq0$ for each local idempotent $e\in R$ (see \cite{15NiYu03}).
A ring }}\textmd{\textit{\textcolor{black}{R}}}\textmd{\textcolor{black}{{}
is called right min-}}\textmd{\textit{\textcolor{black}{PF}}}\textmd{\textcolor{black}{,
if it is a semiperfect, right mininjective, $S_{r}\subseteq^{ess}R_{R}$,
$lr(K)=K$ for every simple left ideal $K\subseteq Re$ for some local
idempotent $e\in R$ (see \cite{15NiYu03}).}}}

\begin{cor}\label{Corollary:(4.18)} Let $R$ be a right ss-injective ring, semiperfect with $S_{r}\subseteq^{ess}R_{R}$.
Then $R$ is right minfull ring and the following statements hold: 

\noindent (1) Every simple right ideal of $R$ is essential in a summand. 

\noindent (2) soc$(eR)$ is simple and essential in $eR$ for every local idempotent $e\in R$. Moreover, $R$ is right finitely cogenerated. 

\noindent (3) For every semisimple right ideal $I$ of $R$, there exists $e=e^{2}\in R$ such that $I\subseteq^{ess}rl(I)\subseteq^{ess}eR$. 

\noindent (4) $S_{r}\subseteq S_{\ell}\subseteq rl(S_{r})$. 

\noindent (5) If $I$  is a semisimple right ideal of $R$ and $aR$ is a simple right ideal of $R$ with $I\cap aR=0$, then $rl(I\oplus aR)=rl(I)\oplus rl(aR)$. 

\noindent (6) $rl(\overset{{\scriptscriptstyle n}}{\underset{{\scriptscriptstyle i=1}}{\bigoplus}}a_{i}R)=\underset{{\scriptscriptstyle i=1}}{\overset{{\scriptscriptstyle n}}{\bigoplus}}rl(a_{i}R)$
, where $\overset{{\scriptscriptstyle n}}{\underset{{\scriptscriptstyle i=1}}{\bigoplus}}a_{i}R$
is a direct sum of simple right ideals. 

\noindent (7) The following statements are equivalent. 

\noindent (a) $S_{r}=rl(S_{r})$. 

\noindent (b) $K=rl(K)$ for every semisimple right ideals $K$ of $R$.

\noindent (c) $kR=rl(kR)$ for every simple right ideals $kR$ of $R$.

\noindent (d) $S_{r}=S_{\ell}$. 

\noindent (e) \emph{soc}$(Re)$ is simple for all local idempotent
$e\in R$. 

\noindent (f) \emph{soc}$(Re)=S_{r}e$ for every local idempotent $e\in R$. 

\noindent (g) $R$ is left mininjective. 

\noindent (h) $L=lr(L)$ for every semisimple left ideals $L$ of $R$.

\noindent (i) $R$ is left minfull ring. 

\noindent (j) $S_{r}\cap J=rl(S_{r}\cap J)$. 

\noindent (k) $K=rl(K)$ for every semisimple small right ideals $K$ of $R$. 

\noindent (l) $L=lr(L)$ for every semisimple small left ideals $L$ of  $R$.

\noindent (8) If $R$ satisfies any condition of (7), then $r(S_{\ell}\cap J)\subseteq^{ess}R_{R}$. \end{cor}
\begin{proof}
\textcolor{black}{(1), (2), (3), (4), (5) and (6) are obtained by
Corollary~\ref{Corollary:(3.11)} and \cite[Theorem 4.12]{2AmYoZe05}.}

\textcolor{black}{(7) The equivalence of (a), (b), (c), (d), (e),
(f), (g), (h) and (i) follows from Corollary~\ref{Corollary:(3.11)} and \cite[Theorem 4.12]{2AmYoZe05}.}

\textcolor{black}{(b)$\Rightarrow$(j) Clear.}

\textcolor{black}{(j)$\Leftrightarrow$(k) By Proposition~\ref{Proposition:(4.6)}(2).}

\textcolor{black}{(k)$\Rightarrow$(c) By Corollary~\ref{Corollary:(4.13)}(1).}

\textcolor{black}{(h)$\Rightarrow$(l) Clear.}

\textcolor{black}{(l)$\Rightarrow$(d) Let $Ra$ be a simple left
ideal of }\textit{\textcolor{black}{R}}\textcolor{black}{. By hypothesis,
$lr(A)=A$ for any simple small left ideal }\textit{\textcolor{black}{A}}\textcolor{black}{{}
of }\textit{\textcolor{black}{R}}\textcolor{black}{. By Lemma~\ref{Lemma:(4.12)},
$lr(A)=A$ for any simple left ideal }\textit{\textcolor{black}{A}}\textcolor{black}{{}
of }\textit{\textcolor{black}{R}}\textcolor{black}{{} and hence $lr(Ra)=Ra$.
Thus }\textit{\textcolor{black}{R}}\textcolor{black}{{} is a right min-PF
ring and it follows from \cite[Theorem 3.14]{14NiYo97}  that $S_{r}=S_{\ell}$.}

\textcolor{black}{(8) Let }\textit{\textcolor{black}{K}}\textcolor{black}{{}
be a right ideal of }\textit{\textcolor{black}{R}}\textcolor{black}{{}
such that $r(S_{\ell}\cap J)\cap K=0$. Then $Kr(S_{\ell}\cap J)=0$
and we have $K\subseteq lr(S_{\ell}\cap J)=S_{\ell}\cap J=S_{r}\cap J$.
Now, $r((S_{\ell}\cap J)+l(K))=$$r(S_{\ell}\cap J)\cap K=0$. Since
}\textit{\textcolor{black}{R}}\textcolor{black}{{} is left Kasch, then
$(S_{\ell}\cap J)+l(K)=R$ by \cite[Corollary 8.28(5)]{10Lam99}. Thus $l(K)=R$
and hence $K=0$, so $r(S_{\ell}\cap J)\subseteq^{ess}R_{R}$.}
\end{proof}

\subparagraph*{\textmd{\textcolor{black}{Recall that }}\textmd{a}\textmd{\textcolor{black}{{}
right }}\textmd{\textit{\textcolor{black}{R}}}\textmd{\textcolor{black}{-module
}}\textmd{\textit{\textcolor{black}{M}}}\textmd{\textcolor{black}{{}
is called almost-injective if $M=E\oplus K$, where }}\textmd{\textit{\textcolor{black}{E}}}\textmd{\textcolor{black}{{}
is injective and }}\textmd{\textit{\textcolor{black}{K}}}\textmd{\textcolor{black}{{}
has zero radical (see \cite{23ZeHuAm11}). After reflect on \cite[Theorem 2.12]{23ZeHuAm11}
we found it is not true always and the reason is due to the homomorphism
$h:(L+J)/J\longrightarrow K$ in the part (3)$\Rightarrow$(1) of the proof of Theorem 2.12 in \cite{23ZeHuAm11} is not well define, in particular see the following example.}}}

\begin{example}\label{Example:(4.19)}
\emph{\textcolor{black}{In particular from the proof of the part (3)$\Rightarrow$(1)
in \cite[Theorem 2.12]{23ZeHuAm11}, we consider $R=\mathbb{Z}_{8}$ and $M=K=<\bar{4}>=\left\{ \bar{0},\bar{4}\right\} $.
Thus $M=E\oplus K$, where $E=0$ is a trivial injective }\textit{\textcolor{black}{R}}\textcolor{black}{-module
and $J(K)=0$. Let $f:L\longrightarrow K$ is the identity map, where
$L=K$. So, the map $h:(L+J)/J\longrightarrow K$ which is given by
$h(\ell+J)=f(\ell)$ is not well define, because $J=\bar{4}+J$ but
$h(J)=f(\bar{0})=\bar{0}\neq\bar{4}=f(\bar{4})=h(\bar{4}+J)$.}}
\end{example}

\subparagraph*{\textmd{The following example shows that there is a   contradiction in \cite[Theorem 2.12]{23ZeHuAm11}.}}

\begin{example}\label{Example:(4.20)}
\emph{\textcolor{black}{Assume that }\textit{\textcolor{black}{R}}\textcolor{black}{{}
is a right artinian ring but not semisimple (this claim is found because
for example $\mathbb{Z}_{8}$ satisfies this property). Now, let }\textit{\textcolor{black}{M}}\textcolor{black}{{}
be a simple right }\textit{\textcolor{black}{R}}\textcolor{black}{-module,
then }\textit{\textcolor{black}{M}}\textcolor{black}{{} is almost-injective.
Clearly, }\textit{\textcolor{black}{R}}\textcolor{black}{{} is semilocal
(see \cite[Theorem 9.2.2]{9Kas82}), thus }\textit{\textcolor{black}{M}}\textcolor{black}{{}
is injective by \cite[Theorem 2.12]{23ZeHuAm11}. Therefore, }\textit{\textcolor{black}{R}}\textcolor{black}{{}
is }\textit{\textcolor{black}{V}}\textcolor{black}{-ring and hence
}\textit{\textcolor{black}{R}}\textcolor{black}{{} is a right semisimple
ring but this contradiction. In other word, Since $\mathbb{Z}_{8}$
is semilocal ring and $<\bar{4}>=\left\{ \bar{0},\bar{4}\right\} $
is almost injective as $\mathbb{Z}_{8}$-module, then $<\bar{4}>$
is injective by \cite[Theorem 2.12]{23ZeHuAm11}. Thus $<\bar{4}>\subseteq^{\oplus}\mathbb{Z}_{8}$
and this contradiction.}}\end{example}

\begin{thm}\label{Theorem:(4.21)} The following statements are equivalent for a ring $R$.

\noindent (1) $R$ is semiprimitive and every almost-injective right $R$-module is quasi-continuous. 

\noindent (2) $R$ is right ss-injective and right minannihilator
ring,  $J$  is   right artinian, and every almost-injective right $R$-module
is quasi-continuous. 

\noindent (3) $R$ is a semisimple ring. \end{thm}
\begin{proof}
\textcolor{black}{(1)$\Rightarrow$(2) and (3)$\Rightarrow$(1) are
clear.}

\textcolor{black}{(2)$\Rightarrow$(3) Let }\textit{\textcolor{black}{M}}\textcolor{black}{{}
be a right }\textit{\textcolor{black}{R}}\textcolor{black}{-module
with zero radical. If }\textit{\textcolor{black}{N}}\textcolor{black}{{}
is an arbitrary nonzero submodule of }\textit{\textcolor{black}{M}}\textcolor{black}{,
then $N\oplus M$ is quasi-continuous and by \cite[Corollary 2.14]{12MoMu90},
}\textit{\textcolor{black}{N}}\textcolor{black}{{} is }\textit{\textcolor{black}{M}}\textcolor{black}{-injective.
Thus $N\leq^{\oplus}M$ and hence }\textit{\textcolor{black}{M}}\textcolor{black}{{}
is semisimple. In particular $R/J$ is semisimple }\textit{\textcolor{black}{R}}\textcolor{black}{-module
and hence $R/J$ is artinian by \cite[Theorem 9.2.2(b)]{9Kas82}, so }\textit{\textcolor{black}{R}}\textcolor{black}{{}
is semilocal ring. Since }\textsl{\textcolor{black}{J}}\textcolor{black}{{}
is a right artinian, then }\textit{\textcolor{black}{R}}\textcolor{black}{{}
is right artinian. So it follows from Corollary~\ref{Corollary:(4.18)}(7) that }\textit{\textcolor{black}{R}}\textcolor{black}{{}
is right and left mininjective. Thus \cite[Corollary 4.8]{14NiYo97} implies
that }\textit{\textcolor{black}{R}}\textcolor{black}{{} is }\textit{\textcolor{black}{QF}}\textcolor{black}{{}
ring. By hypothesis, $R\oplus(R/J)$ is quasi-continuous (since }\textit{\textcolor{black}{R}}\textcolor{black}{{}
is self-injective), so again by \cite[Corollary 2.14]{12MoMu90} we have that
$R/J$ is injective. Since }\textit{\textcolor{black}{R}}\textcolor{black}{{}
is }\textit{\textcolor{black}{QF}}\textcolor{black}{{} ring, then $R/J$
is projective (see \cite[Theorem 13.6.1]{9Kas82}). Thus the canonical map
$\pi:R\longrightarrow R/J$ is splits and hence $J\leq^{\oplus}R$,
that is $J=0$. Therefore }\textit{\textcolor{black}{R}}\textcolor{black}{{}
is semisimple. }
\end{proof}

\section{STRONGLY SS-INJECTIVE RINGS}

\begin{prop}\label{Proposition:(5.1)} A ring $R$ is strongly right ss-injective if and only if every finitely generated projective right $R$-module is strongly ss-injective. \end{prop}
\begin{proof}
Since a finite direct sum of strongly ss-injective modules is strongly
ss-injective, so every finitely generated free right $R$-module is strongly
ss-injective. But a direct summand of strongly ss-injective is strongly
ss-injective. Therefore, every finitely generated projective is strongly
ss-injective. The converse is clear.
\end{proof}

\subparagraph*{\textmd{\textcolor{black}{A ring }}\textmd{\textit{\textcolor{black}{R}}}\textmd{\textcolor{black}{{}
is called a right Ikeda-Nakayama ring if $l(A\cap B)=l(A)+l(B)$ for
all right ideals }}\textmd{\textit{\textcolor{black}{A}}}\textmd{\textcolor{black}{{}
and }}\textmd{\textit{\textcolor{black}{B}}}\textmd{\textcolor{black}{{}
of }}\textmd{\textit{\textcolor{black}{R}}}\textmd{\textcolor{black}{{}
(see \cite[p.148]{15NiYu03}). In the next proposition, the strongly ss-injectivity
gives a new version of Ikeda-Nakayama rings.}}}

\begin{prop}\label{Proposition:(5.2)} Let $R$ be a strongly right ss-injective ring, then $l(A\cap B)=l(A)+l(B)$  for
all semisimple small right ideals $A$ and all right ideals $B$ of $R$.\end{prop}
\begin{proof}
\textcolor{black}{Let $x\in l(A\cap B)$ and define $\alpha:A+B\longrightarrow R_{R}$
by $\alpha(a+b)=xa$ for all $a\in A$ and $b\in B$. Clearly, $\alpha$
is well define, because if $a_{1}+b_{1}=a_{2}+b_{2}$, then $a_{1}-a_{2}=b_{2}-b_{1}$,
that is $x(a_{1}-a_{2})=0$, so $\alpha(a_{1}+b_{1})=\alpha(a_{2}+b_{2})$.
The map $\alpha$ induces an }\textit{\textcolor{black}{R}}\textcolor{black}{-homomorphism
$\tilde{\alpha}:(A+B)/B\longrightarrow R_{R}$ which is given by $\tilde{\alpha}(a+B)=xa$
for all $a\in A$. Since $(A+B)/B\subseteq$  soc$(R/B)\cap J(R/B)$ and
}\textit{\textcolor{black}{R}}\textcolor{black}{{} is a strongly right
ss-injective, $\tilde{\alpha}$ can be extended to an }\textit{\textcolor{black}{R}}\textcolor{black}{-homomorphism
$\gamma:R/B\longrightarrow R_{R}$. If $\gamma(1+B)=y$, for some $y\in R$,
then $y(a+b)=xa$, for all $a\in A$ and $b\in B$. In particular,
$ya=xa$ for all $a\in A$ and $yb=0$ for all $b\in B$. Hence $x=(x-y)+y\in l(A)+l(B)$.
Therefore, $l(A\cap B)\subseteq l(A)+l(B)$. Since the converse is
always holds, thus the proof is complete.}
\end{proof}

\subparagraph*{\textmd{\textcolor{black}{Recall that a ring }}\textmd{\textit{\textcolor{black}{R}}}\textmd{\textcolor{black}{{}
is said to be right simple }}\textmd{\textit{\textcolor{black}{J}}}\textmd{\textcolor{black}{-injective
if for any small right ideal }}\textmd{\textit{\textcolor{black}{I}}}\textmd{\textcolor{black}{{}
and any }}\textmd{\textit{\textcolor{black}{R}}}\textmd{\textcolor{black}{-homomorphism
$\alpha:I\longrightarrow R_{R}$ with simple image, $\alpha=c.$ for
some $c\in R$ (see \cite{21YoZh04}).}}}

\begin{cor}\label{Corollary:(5.3)} Every strongly right ss-injective ring is right simple $J$-injective. \end{cor}
\begin{proof} By Proposition~\ref{Proposition:(3.1)}.\end{proof}

\begin{rem}\label{Remark:(5.4)}\emph{ The converse of Corollary~\ref{Corollary:(5.3)}   is not true (see Example~\ref{Example:(5.7)})}.\end{rem}

\begin{prop}\label{Proposition:(5.5)} Let $R$ be a right Kasch and strongly right ss-injective ring. Then: 

\noindent (1) $rl(K)=K$, for every small right ideal $K$. Moreover, $R$ is right minannihilator. 

\noindent (2) If $R$ is left Kasch, then $r(J)\subseteq^{ess}R_{R}$. \end{prop}
\begin{proof}
\textcolor{black}{(1) By Corollary~\ref{Corollary:(5.3)}  and \cite[Lemma 2.4]{21YoZh04}.}

\textcolor{black}{(2) Let }\textit{\textcolor{black}{K}}\textcolor{black}{{}
be a right ideal of }\textit{\textcolor{black}{R}}\textcolor{black}{{}
and $r(J)\cap K=0$. Then $Kr(J)=0$ and we obtain $K\subseteq lr(J)=J$,
because }\textit{\textcolor{black}{R}}\textcolor{black}{{} is left Kasch.
By (1), we have $r(J+l(K))=r(J)\cap K=0$ and this means that $J+l(K)=R$
(since }\textit{\textcolor{black}{R}}\textcolor{black}{{} is left Kasch).
Thus $K=0$ and hence $r(J)\subseteq^{ess}R_{R}$.}
\end{proof}

\subparagraph*{\textmd{The following examples show that the  classes of rings:  strongly
ss-injective rings, soc-injective rings and of small injective rings
are different.}}

\begin{example}\label{Example:(5.6)}
\emph{\textcolor{black}{Let $R=\mathbb{Z}_{(p)}=\{\frac{m}{n}\mid p$ does not
divide $n\}$, the localization ring of $\mathbb{Z}$ at the prime
$p$. Then }\textit{\textcolor{black}{R}}\textcolor{black}{{} is a commutative
local ring and it has zero socle but not principally small injective
(see \cite[Example 4]{20Xia11}). Since $S_{r}=0$, thus }\textit{\textcolor{black}{R}}\textcolor{black}{{}
is strongly soc-injective ring and hence }\textit{\textcolor{black}{R}}\textcolor{black}{{}
is strongly ss-injective ring.}}
\end{example}

\begin{example}\label{Example:(5.7)}
 \emph{Let $R=\left\{ \begin{array}{cc}
\left(\begin{array}{cc}
n & x\\
0 & n\end{array}\right)\mid & n\in\mathbb{Z}, \, x\in\mathbb{Z}_{2}\end{array}\right\} $. Thus $R$  is a commutative ring, $J=S_{r}=\left\{ \begin{array}{cc}
\left(\begin{array}{cc}
0 & x\\
0 & 0\end{array}\right)\mid & x\in\mathbb{Z}_{2}\end{array}\right\} $ and $R$  is small injective
(see \cite[Example(i)]{19ThQu09}). Let $A=J$ and
}
\emph{\noindent $B=\left\{ \begin{array}{cc}
\left(\begin{array}{cc}
2n & 0\\
0 & 2n\end{array}\right)\mid & n\in\mathbb{Z}\end{array}\right\} $, then $l(A)=\left\{ \begin{array}{cc}
\left(\begin{array}{cc}
2n & y\\
0 & 2n\end{array}\right)\mid & n\in\mathbb{Z},\,y\in\mathbb{Z}_{2}\end{array}\right\} $ and
}

\noindent\emph{ $l(B)=\left\{ \begin{array}{cc}
\left(\begin{array}{cc}
0 & y\\
0 & 0\end{array}\right)\mid & y\in\mathbb{Z}_{2}\end{array}\right\} $. Thus $l(A)+l(B)=\left\{ \begin{array}{cc}
\left(\begin{array}{cc}
2n & y\\
0 & 2n\end{array}\right)\mid & n\in\mathbb{Z},\, y\in\mathbb{Z}_{2}\end{array}\right\} $.}

\noindent\emph{ Since $A\cap B=0$, thus $l(A\cap B)=R$ and this implies that $l(A)+l(B)\neq l(A\cap B)$.
Therefore $R$  is not strongly ss-injective and not strongly soc-injective by Proposition~\ref{Proposition:(5.2)}.}
\end{example}

\begin{example}\label{Example:(5.8)}
\emph{\textcolor{black}{Let $F=\mathbb{Z}_{2}$ be the field of two elements,
$F_{i}=F$ for $i=1,2,3,...$, $Q=\overset{{\scriptscriptstyle \infty}}{\underset{{\scriptscriptstyle i=1}}{\prod}}F_{i}$,
$S=\overset{{\scriptscriptstyle \infty}}{\underset{{\scriptscriptstyle i=1}}{\bigoplus}}F_{i}$
. If }\textit{\textcolor{black}{R}}\textcolor{black}{{} is the subring
of }\textit{\textcolor{black}{Q}}\textcolor{black}{{} generated by 1
and }\textit{\textcolor{black}{S}}\textcolor{black}{, then }\textit{\textcolor{black}{R}}\textcolor{black}{{}
is a Von Neumann regular ring (see \cite[Example (1), p.28]{22Zey14}). Since
}\textit{\textcolor{black}{R}}\textcolor{black}{{} is commutative, thus
every simple }\textit{\textcolor{black}{R}}\textcolor{black}{-module
is injective by \cite[Corollary 3.73]{10Lam99}. Thus }\textit{\textcolor{black}{R}}\textcolor{black}{{}
is }\textit{\textcolor{black}{V}}\textcolor{black}{-ring and hence
$J(N)=0$ for every right }\textit{\textcolor{black}{R}}\textcolor{black}{-module
}\textit{\textcolor{black}{N}}\textcolor{black}{. It follows from
Corollary~\ref{Corollary:(3.9)}   that every   }\textit{\textcolor{black}{R}}\textcolor{black}{-module
is strongly ss-injective. In particular, }\textit{\textcolor{black}{R}}\textcolor{black}{{}
is strongly ss-injective ring. But }\textit{\textcolor{black}{R}}\textcolor{black}{{}
is not soc-injective (see \cite[Example (1)]{22Zey14}).}}
\end{example}

\begin{example}\label{Example:(5.9)}
\emph{\textcolor{black}{Let $R=\mathbb{Z}_{2}[x_{1},x_{2},...]$ where $\mathbb{Z}_{2}$
is the field of two elements, $x_{i}^{3}=0$ for all i, $x_{i}x_{j}=0$
for all $i\neq j$ and $x_{i}^{2}=x_{j}^{2}\neq0$ for all i and j.
If $m=x_{i}^{2}$, then }\textit{\textcolor{black}{R}}\textcolor{black}{{}
is a commutative, semiprimary, local, soc-injective ring with $J=$span\{m, $x_{1}$, $x_{2}$,
... \}, and }\textit{\textcolor{black}{R}}\textcolor{black}{{} has simple
essential socle $J^{2}=\mathbb{Z}_{2}m$ (see \cite[Example 5.7]{2AmYoZe05}).
It follows from \cite[Example 5.7]{2AmYoZe05} that the }\textit{\textcolor{black}{R}}\textcolor{black}{-homomorphism
$\gamma:J\longrightarrow R$ which is given by $\gamma(a)=a^{2}$
for all $a\in J$ with simple image can be not extended to }\textit{\textcolor{black}{R}}\textcolor{black}{,
then }\textit{\textcolor{black}{R}}\textcolor{black}{{} is not simple
}\textit{\textcolor{black}{J}}\textcolor{black}{-injective and not
small injective, so it follows from Corollary~\ref{Corollary:(5.3)}   that }\textit{\textcolor{black}{R}}\textcolor{black}{{}
is not strongly ss-injective.}}
\end{example}

\subparagraph*{\textmd{\textcolor{black}{Recall that }}\textmd{\textit{\textcolor{black}{R}}}\textmd{\textcolor{black}{{}
is said to be right minsymmetric ring if $aR$ is simple right ideal
then $Ra$ is simple left ideal (see \cite{14NiYo97}). Every right mininjective
ring is right minsymmetric by \cite[Theorem 1.14]{14NiYo97}.}}}

\begin{thm}\label{Theorem:(5.10)} A ring $R$ is  QF  if and only if $R$ is a strongly right
ss-injective and right noetherian ring with $S_{r}\subseteq^{ess}R_{R}$.\end{thm}
\begin{proof}
\textcolor{black}{($\Rightarrow$) This is clear.}

\textcolor{black}{($\Leftarrow$) By Corollary~\ref{Corollary:(4.3)}(1), }\textit{\textcolor{black}{R}}\textcolor{black}{{}
is right minsymmetric. It follows from \cite[Lemma 2.2]{19ThQu09} that }\textit{\textcolor{black}{R}}\textcolor{black}{{}
is right perfect. Thus }\textit{\textcolor{black}{R}}\textcolor{black}{{}
is strongly right soc-injective, by Theorem~\ref{Theorem:(3.13)}. Since $S_{r}\subseteq^{ess}R_{R}$,
so it follows from \cite[Corollary 3.2]{2AmYoZe05} that }\textit{\textcolor{black}{R}}\textcolor{black}{{}
is self-injective and hence }\textit{\textcolor{black}{R}}\textcolor{black}{{}
is }\textit{\textcolor{black}{QF}}\textcolor{black}{.}\end{proof}

\begin{cor}\label{Corollary:(5.11)} For a ring $R$ the following statements are true.

\noindent (1) $R$ is semisimple if and only if  $S_{r}\subseteq^{ess}R_{R}$
and every semisimple right $R$-module
is strongly soc-injective.

\noindent (2) $R$ is  QF  if and only if $R$ is
strongly right ss-injective, semiperfect
with essential right socle and $R/S_{r}$ is noetherian as right $R$-module. \end{cor}
\begin{proof}
(1) Suppose that $S_{r}\subseteq^{ess}R_{R}$ and
every semisimple right $R$-module
is strongly soc-injective, then $R$
is a right noetherian right V-ring by \cite[Proposition 3.12]{2AmYoZe05}, so
it follows from Corollary~\ref{Corollary:(3.9)} that $R$
is strongly right ss-injective. Thus $R$
is  QF  by Theorem~\ref{Theorem:(5.10)}.
But $J=0$, so $R$  is
semisimple. The converse is clear.

(2) By\textcolor{black}{{} \cite[Theorem 2.9]{14NiYo97}, $J=Z_{r}$. Since
$R/Z_{2}^{r}$ is a homomorphic image of $R/Z_{r}$ and }\textit{\textcolor{black}{R}}\textcolor{black}{{}
is a semilocal ring, thus }\textit{\textcolor{black}{R}}\textcolor{black}{{}
is a right }\textit{\textcolor{black}{t}}\textcolor{black}{-semisimple.
By Corollary~\ref{Corollary:(3.19)}, }\textit{\textcolor{black}{R}}\textcolor{black}{{}
is right noetherian, so it follows from Theorem~\ref{Theorem:(5.10)} that }\textit{\textcolor{black}{R}}\textcolor{black}{{}
is }\textit{\textcolor{black}{QF}}\textcolor{black}{. The converse
is clear.}\end{proof}

\begin{thm}\label{Theorem:(5.12)} A ring $R$ is $QF$  if and only if $R$ is a strongly right
ss-injective, $l(J^{2})$ is a countable generated left ideal, $S_{r}\subseteq^{ess}R_{R}$
and the chain $r(x_{1})\subseteq r(x_{2}x_{1})\subseteq...\subseteq r(x_{n}x_{n-1}...x_{1})\subseteq...$
terminates for every infinite sequence $x_{1},x_{2},...$ in $R$.\end{thm}
\begin{proof}
\textcolor{black}{($\Rightarrow$) Clear.}

\textcolor{black}{($\Leftarrow$) By \cite[Lemma 2.2]{19ThQu09}, }\textit{\textcolor{black}{R}}\textcolor{black}{{}
is right perfect. Since $S_{r}\subseteq^{ess}R_{R}$, thus }\textit{\textcolor{black}{R}}\textcolor{black}{{}
is right Kasch (by \cite[Theorem 3.7]{14NiYo97}). Since }\textit{\textcolor{black}{R}}\textcolor{black}{{}
is strongly right ss-injective, thus }\textit{\textcolor{black}{R}}\textcolor{black}{{}
is right simple }\textit{\textcolor{black}{J}}\textcolor{black}{-injective,
by Corollary~\ref{Corollary:(5.3)}. Now, by Proposition~\ref{Proposition:(5.5)}(1) we have $rl(S_{r}\cap J)=S_{r}\cap J$,
so it follows from Corollary~\ref{Corollary:(4.18)}(7) that $S_{r}=S_{\ell}$. By \cite[Lemma 3.36]{15NiYu03}, $S_{2}^{r}=l(J^{2})$. The result now follows from
\cite[Theorem 2.18]{21YoZh04}.}\end{proof}

\begin{rem}\label{Remark:(5.13)}
\emph{The condition \textcolor{black}{$S_{r}\subseteq^{ess}R_{R}$ in Theorem~\ref{Theorem:(5.10)}
  and Theorem~\ref{Theorem:(5.12)}  can be not deleted,  for example, $\mathbb{Z}$ is
strongly ss-injective noetherian ring but not }\textit{\textcolor{black}{QF}}\textcolor{black}{.}}\end{rem}

The following two results are extension of Proposition 5.8 in \cite{2AmYoZe05}.

\begin{cor}\label{Corollary:(5.15)} The following statements are equivalent. 

\noindent (1) $R$ is a $QF$  ring. 

\noindent (2) $R$ is a left perfect, strongly left and right ss-injective ring. \end{cor}
\begin{proof}
By Corollary~\ref{Corollary:(5.3)}   and \cite[Corollary 2.12]{21YoZh04}.\end{proof}

\begin{thm}\label{Theorem:(5.16)} The following statements are equivalent: 

\noindent (1) $R$ is a $QF$  ring. 

\noindent (2) $R$ is a strongly left and right ss-injective, right Kasch and $J$ is left $t$-nilpotent. 

\noindent (3) $R$ is a strongly left and right ss-injective, left Kasch and $J$ is left $t$-nilpotent. \end{thm}
\begin{proof}
(1)\textcolor{black}{$\Rightarrow$(2) and (1)$\Rightarrow$(3) are
clear.}

\textcolor{black}{(3)$\Rightarrow$(1) Suppose that $xR$ is simple
right ideal. Thus either $rl(x)=xR\subseteq^{\oplus}R_{R}$ or $x\in J$.
If $x\in J$, then $rl(x)=xR$ (since }\textit{\textcolor{black}{R}}\textcolor{black}{{}
is right minannihilator), so it follows from Theorem~\ref{Theorem:(3.4)}  that $rl(x)\subseteq^{ess}E\subseteq^{\oplus}R_{R}$.
Therefore, $rl(x)$ is essential in a direct summand of $R_{R}$ for
every simple right ideal $xR$. Let }\textit{\textcolor{black}{K}}\textcolor{black}{{}
be a maximal left ideal of }\textit{\textcolor{black}{R}}\textcolor{black}{.
Since }\textit{\textcolor{black}{R}}\textcolor{black}{{} is left Kasch,
thus $r(K)\neq0$ by \cite[Corollary 8.28]{10Lam99}. Choose $0\neq y\in r(K)$,
so $K\subseteq l(y)$ and we conclude that $K=l(y)$. Since $Ry\cong R/l(y)$,
thus $Ry$ is simple left ideal. But }\textit{\textcolor{black}{R}}\textcolor{black}{{}
is left mininjective ring, so $yR$ is right simple ideal by \cite[Theorem 1.14]{14NiYo97} and this implies that $r(K)\subseteq^{ess}eR$ for
some $e^{2}=e\in R$ (since $r(K)=rl(y)$). Thus }\textit{\textcolor{black}{R}}\textcolor{black}{{}
is semiperfect by \cite[Lemma 4.1]{15NiYu03} and hence }\textit{\textcolor{black}{R}}\textcolor{black}{{}
is left perfect (since }\textit{\textcolor{black}{J}}\textcolor{black}{{}
is left }\textit{\textcolor{black}{t}}\textcolor{black}{-nilpotent),
so it follows from Corollary~\ref{Corollary:(5.15)}   that }\textit{\textcolor{black}{R}}\textcolor{black}{{}
is }\textit{\textcolor{black}{QF}}\textcolor{black}{.}

\textcolor{black}{(2)$\Rightarrow$(1) It is similar to the proof
of (3)$\Rightarrow$(1).}\end{proof}

\begin{thm}\label{Theorem:(5.17)} The ring $R$ is $QF$  if and only if $R$ is   strongly left and right ss-injective, left and right Kasch, and the chain $l(a_{1})\subseteq l(a_{1}a_{2})\subseteq l(a_{1}a_{2}a_{3})\subseteq...$ terminates for every $a_{1},a_{2},...\in Z_{\ell}$. \end{thm}
\begin{proof}
\textcolor{black}{($\Rightarrow$) Clear.}

\textcolor{black}{($\Leftarrow$) By Proposition~\ref{Proposition:(5.5)}, $l(J)$ is essential
in $_{R}R$. Thus $J\subseteq Z_{\ell}$. Let $a_{1},a_{2},...\in J$
, we have $l(a_{1})\subseteq l(a_{1}a_{2})\subseteq l(a_{1}a_{2}a_{3})\subseteq...$.
Thus there exists $k\in\mathbb{N}$ such that $l(a_{1}...a_{k})=l(a_{1}...a_{k}a_{k+1})$
(by hypothesis). Suppose that $a_{1}...a_{k}\neq0$, so $R(a_{1}...a_{k})\cap l(a_{k+1})\neq0$
(since $l(a_{k+1})$ is essential in $_{R}R$). Thus $ra_{1}...a_{k}\neq0$
and $ra_{1}...a_{k}a_{k+1}=0$ for some $r\in R$, a contradiction.
Therefore, $a_{1}...a_{k}=0$ and hence }\textit{\textcolor{black}{J}}\textcolor{black}{{}
is left }\textit{\textcolor{black}{t}}\textcolor{black}{-nilpotent,
so it follows from Theorem~\ref{Theorem:(5.16)}   that }\textit{\textcolor{black}{R}}\textcolor{black}{{}
is }\textit{\textcolor{black}{QF}}\textcolor{black}{.}\end{proof}

\begin{cor}\label{Corollary:(5.18)} The ring $R$ is $QF$  if and only if $R$ is   strongly left and right ss-injective with essential right socle, and the chain $r(a_{1})\subseteq r(a_{2}a_{1})\subseteq r(a_{3}a_{2}a_{1})\subseteq...$
terminates for every infinite sequence $a_{1},a_{2},...$ in $R$.\end{cor}
\begin{proof}
By \cite[Lemma 2.2]{19ThQu09} and Corollary~\ref{Corollary:(5.15)}.
\end{proof}

\end{document}